\documentclass[a4paper,10pt]{article}

\usepackage{amsmath,amssymb,amsfonts,amsthm}
\usepackage{graphicx}
\usepackage{color}

\DeclareMathOperator{\JL}{\mathsf{J}}
\DeclareMathOperator{\DL}{\mathsf{D}}
\DeclareMathOperator{\JM}{\mathsf{J}_{\mathcal{M}}}
\DeclareMathOperator{\DM}{\mathsf{D}_{\mathcal{M}}}
\DeclareMathOperator{\JK}{\mathsf{J}_{\mathcal{K}}}
\DeclareMathOperator{\DK}{\mathsf{D}_{\mathcal{K}}}

\DeclareMathOperator{\FD}{\mathcal{D}_{\mathit{f}}}
\DeclareMathOperator{\M}{\mathcal{M}}
\DeclareMathOperator{\BDL}{\mathcal{DL}}
\DeclareMathOperator{\FM}{\mathcal{M}_{\mathit{f}}}
\DeclareMathOperator{\FK}{\mathcal{K}_{\mathit{f}}}
\DeclareMathOperator{\K}{\mathcal{K}}
\DeclareMathOperator{\FP}{\mathcal{P}_{\mathit{f}}}
\DeclareMathOperator{\FPM}{\mathcal{PM}_{\mathit{f}}}
\DeclareMathOperator{\FPK}{\mathcal{PK}_{\mathit{f}}}
\DeclareMathOperator{\V}{\mathcal{V}}

\DeclareMathOperator{\D}{\mathbf{D}}
\DeclareMathOperator{\A}{\mathbf{A}}

\begin{document}

\newtheorem{theorem}{Theorem}
\newtheorem{algorithm}[theorem]{Algorithm}
\newtheorem{fact}[theorem]{Fact}
\newtheorem{problem}[theorem]{Problem}
\newtheorem{claim}[theorem]{Claim}
\newtheorem{lemma}[theorem]{Lemma}
\newtheorem{example}[theorem]{Example}
\newtheorem{corollary}[theorem]{Corollary}
\newtheorem{proposition}[theorem]{Proposition}
\newtheorem{notation}[theorem]{Notation}
\newtheorem{remark7}[theorem]{Remark}
\newtheorem{conjecture}[theorem]{Conjecture}
\theoremstyle{definition}
\newtheorem{definition}[theorem]{Definition}
\newtheorem{remark}[theorem]{Remark}

\title{Unification and Projectivity\\in De Morgan and Kleene Algebras}
\author{Simone Bova and Leonardo Cabrer}
\date{}

\maketitle

\begin{abstract}
We provide a complete classification of solvable instances of the equational unification problem over De Morgan and Kleene algebras 
with respect to unification type.  The key tool is a combinatorial characterization of finitely generated projective De Morgan and Kleene algebras.  
% (in particular, we establish that De Morgan and Kleene algebras have nullary unification type).  
%Relying on this results, we will prove that the unification type of the equational classes of De Morgan and Kleene algebras is nullary and we will provide a classification procedure to determine the unification type of each particular unification problem in this classes.
\end{abstract}

\section{Introduction}

% A \emph{De Morgan} algebra $\mathbf{A}=(A,\wedge,\vee,',0,1)$ is a bounded distributive lattice 
% with an involution satisfying De Morgan laws, that is a unary operation %$x'$ 
% satisfying $x=x''$ and $x \wedge y=(x' \vee y')'$.  
% A \emph{Kleene} algebra is a De Morgan algebra satisfying $x\wedge x' \leq y \vee y'$.  
% A \emph{Boolean} algebra is a Kleene algebra satisfying $x\wedge x'=0$.  
A \emph{De Morgan} algebra 
is a bounded distributive lattice with an involution satisfying De Morgan laws;  
that is, a unary operation satisfying $x=x''$ and $x \wedge y=(x' \vee y')'$.  
A \emph{Kleene} algebra is a De Morgan algebra satisfying $x\wedge x' \leq y \vee y'$.  
%A \emph{Boolean} algebra is a Kleene algebra satisfying $x\wedge x'=0$.  
% DO WE NEED IT HERE?  If $x \in A$ is such that $x=x'$, we call $x$ a fixpoint.  
% We let $\mathbf{2}=(\{0,1\},\wedge,\vee,',0,1)$ denote the De Morgan (in fact, Boolean) algebra of two elements with no fixpoints, 
% $\mathbf{3}=(\{0,2,1\},\wedge,\vee,',0,1)$ denote the De Morgan (in fact, Kleene) algebra of three elements with one fixpoint, 
% $\mathbf{4}=(\{0,2,3,1\},\wedge,\vee,',0,1)$ denote the De Morgan algebra of four elements with two fixpoints.  
In \cite{K58}, Kalman shows that 
% a nontrivial De Morgan algebra $\mathbf{A}$ 
% is subdirectly irreducible iff $\mathbf{A}$ is isomorphic to either $\mathbf{2}$, $\mathbf{3}$, or $\mathbf{4}$.  
% As a consequence, 
the lattice of (nontrivial) varieties of De Morgan algebras is a three-element chain 
formed by Boolean algebras (Kleene algebras satisfying $x\wedge x'=0$), 
Kleene algebras, and De Morgan algebras, 
and these variety are locally finite; that is, finite, finitely presented, 
and finitely generated algebras coincide.  

In a variety of algebras, the symbolic (equational) unification problem is the problem of solving finite 
systems of equations over free algebras.  An instance of the symbolic unification problem is a finite system of equations, 
and a solution (a unifier) is an assignment of the variables to terms such that the system holds identically in the variety.  
The set of unifiers of a solvable instance supports a natural order, 
and the instances are classified depending on the properties of their maximal unifiers.  
In this paper, we provide a complete (first-order, decidable) classification of solvable instances 
of the unification problem over De Morgan and Kleene algebras with respect to their unification type.  

The key tool towards the classification is a combinatorial (first-order, decidable) characterization 
of finitely generated projective De Morgan and Kleene algebras, motivated by 
the nice theory of algebraic (equational) unification introduced by Ghilardi \cite{G97}.  
In the algebraic unification setting, 
an instance of the unification problem is a finitely presented algebra in a certain variety, 
a unifier is a homomorphism to a finitely presented projective algebra in the variety, 
and unifiers support a natural order that determines the unification type of the instance, 
in such a way that it coincides with the unification type of its finite presentation, 
viewed as an instance of the symbolic unification problem.

Even if projective Boolean algebras have been characterized in \cite{BH70,S51}, 
a complete characterization of projective De Morgan and Kleene algebras lacks in the literature.  
% In the effective application of the algebraic unification, 
% en explicit characterization of finitely presented projective algebras is necessary.  
% Unfortunately, a characterization of projective De Morgan and Kleene algebras lacks, 
% even in the finite case (projective Boolean algebras have been characterized in \cite{BH70,S51}).    
% (injective and) projective Boolean algebras have been characterized in \cite{BH70,S51}, 
% and in \cite{C75}, Cignoli characterizes injective De Morgan and Kleene algebras.  
In this note, also motivated by an effective application of the algebraic unification framework, 
we initiate the study of projective De Morgan and Kleene algebras, and relying on finite duality theorems \cite{CF77}, 
we provide a combinatorial characterization of finitely generated projective algebras, 
and we exploit it to classify all solvable instances of the equational unification problem 
over De Morgan and Kleene algebras with respect to their unification type; 
in particular, we establish that De Morgan and Kleene algebras 
have nullary equational unification type (and avoid the infinitary type). 

The paper is organized as follows.  In Section~\ref{sect:prel}, 
we collect from the literature the background on projective algebras, duality theory, 
and unification theory necessary for the rest of the paper.  For standard undefined notions and facts 
in order theory, universal algebra, category theory, and unification theory, we refer the reader to 
\cite{DP02}, \cite{MMT87}, \cite{ML98}, and \cite{BS01} respectively.  In Section~\ref{sect:fpdmk}, 
we introduce the characterization of finite projective De Morgan and Kleene algebras.  
In Section~\ref{sect:fpdmk}, we introduce the characterization of finite projective De Morgan and Kleene algebras 
(respectively, Theorem~\ref{th:mmain} and Theorem~\ref{th:kmain}).  
In Section~\ref{sect:unif}, we provide complete classification with respect to unification type 
of all solvable instances of the equational unification problem 
over bounded distributive lattices, Kleene algebras, and De Morgan algebras 
(respectively, Theorem~\ref{Thm:UnifLat}, Theorem~\ref{Theo:UnifClassKleene}, 
and Theorem~\ref{Theo:UnifClassDeMorgan}).  The distributive lattices case 
tightens previous results by Ghilardi \cite{G97}, and outlines the key ideas 
involved in the study of the more demanding cases of Kleene and De Morgan algebras.

\section{Preliminaries}\label{sect:prel} 

%For undefined notions on ordered sets and lattices, we refer the reader to \cite{DP02}.  
Let $\mathbf{P}=(P,\leq)$ be a preorder, that is, $\leq$ is reflexive and transitive.  
If $x$ and $y$ are incomparable in $\mathbf{P}$, we write $x \parallel y$.  Given $X,Y \subseteq P$, we write $X \leq Y$ iff 
$x \leq y$ for all $x \in X$ and $y \in Y$; we freely omit brackets, 
writing for instance $x \leq y,z$ instead of $\{x\} \leq \{y,z\}$.  
If $X \subseteq P$, 
we denote by $(X]$ and $[X)$ respectively the \emph{downset} and \emph{upset} in $P$ 
generated by $X$, namely $(X]=\{y \in P \mid \text{$y \leq x$ for some $x \in X$} \}$ and
$[X)=\{y \in P \mid \text{$y \geq x$ for some $x \in X$} \}$; if $X=\{x\}$ we freely write $(x]$ and $[x)$.  
If $x,y \in P$, we write $[x,y]=\{ z \in P \mid x \leq z \leq y \}$.  
A set $X \subseteq P$ is \emph{directed} if for all $x,y \in X$ there exists $z \in X$ such that $x,y \leq z$.   
We denote minimal elements in $\mathbf{P}$ by $\mathrm{min}(\mathbf{P})=\{ x \in P \mid \text{$y \leq x$ implies $x \leq y$ for all $y \in P$} \}$.
%; 
%if $\mathbf{P}$ is a poset, then $\mathrm{min}(\mathbf{P})=\{ x \in P \mid \text{$x \leq y$ for all $y \in %P$} \}$.  
Similarly we denote maximal elements in $\mathbf{P}$ by $\mathrm{max}(\mathbf{P})$.  
Let $\mathbf{P}=(P,\leq)$ and $\mathbf{Q}=(Q,\leq)$ be preorders.  A map $f \colon P \to Q$ 
is \emph{monotone} if $x \leq y$ implies $f(x) \leq f(y)$.  

\subsection{Projective Algebras}

%For background on universal algebra, we refer the reader to \cite{MMT87}.

Let $\V$ be a variety of algebras and $\kappa$ be an arbitrary cardinal. An algebra $\mathbf{B} \in \V$
is said to
have the \emph{universal mapping} property for $\kappa $ if there exists  $X \subseteq B$ such that $|X|=\kappa$ and for every $\mathbf{A} \in V$,
and every map $f \colon X \to \mathbf{A}$ there exists a (unique) homomorphism
$g \colon \mathbf{B} \to \mathbf{A}$ extending $f$ (any $x \in X$ is said a \emph{free generator}, 
and $\mathbf{B}$ is said \emph{freely generated} by $X$).  
For every cardinal $\kappa$, there exists a unique algebra with the universal mapping property freely generated by a set of cardinality $\kappa$, 
called the \emph{free} $\kappa$-generated algebra in $\V$, and denoted by $\mathbf{F}_{\V}(\kappa)$.  

Since the varieties of De Morgan and Kleene algebras, in symbols $\M$ and $\K$ respectively, 
are generated by single finite algebras \cite{K58}, they are locally finite, that is, 
finitely generated and finite algebras coincide.  

\begin{example}\label{example:freedm1}
By direct computation, $\mathbf{F}_{\M}(1)$ is the bounded distributive lattice over $\{0,x \wedge x',x,x',x \vee x',1\}$
shown in Figure \ref{Fig:FM1}.
\begin{figure}[h]
\centering
\begin{picture}(0,0)%
\includegraphics{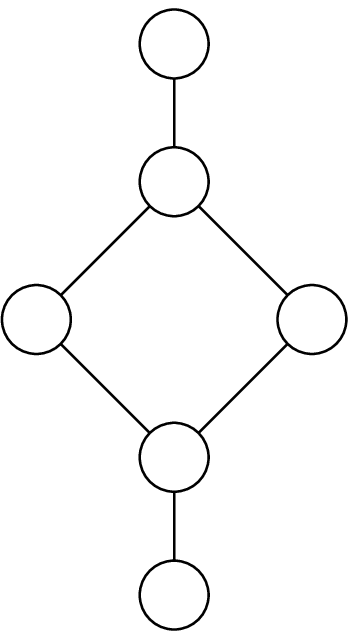}%
\end{picture}%
\setlength{\unitlength}{2901sp}%
\begingroup\makeatletter\ifx\SetFigFont\undefined%
\gdef\SetFigFont#1#2#3#4#5{%
  \reset@font\fontsize{#1}{#2pt}%
  \fontfamily{#3}\fontseries{#4}\fontshape{#5}%
  \selectfont}%
\fi\endgroup%
\begin{picture}(2280,4080)(15061,-5701)
\put(16199,-5516){\makebox(0,0)[b]{\smash{{\SetFigFont{6}{7.2}{\familydefault}{\mddefault}{\updefault}{\color[rgb]{0,0,0}$0$}%
}}}}
\put(17114,-3703){\makebox(0,0)[b]{\smash{{\SetFigFont{6}{7.2}{\familydefault}{\mddefault}{\updefault}{\color[rgb]{0,0,0}$x'$}%
}}}}
\put(15297,-3697){\makebox(0,0)[b]{\smash{{\SetFigFont{6}{7.2}{\familydefault}{\mddefault}{\updefault}{\color[rgb]{0,0,0}$x$}%
}}}}
\put(16198,-1895){\makebox(0,0)[b]{\smash{{\SetFigFont{6}{7.2}{\familydefault}{\mddefault}{\updefault}{\color[rgb]{0,0,0}$1$}%
}}}}
\put(16196,-2795){\makebox(0,0)[b]{\smash{{\SetFigFont{5}{6.0}{\familydefault}{\mddefault}{\updefault}{\color[rgb]{0,0,0}$x\vee x'$}%
}}}}
\put(16200,-4607){\makebox(0,0)[b]{\smash{{\SetFigFont{5}{6.0}{\familydefault}{\mddefault}{\updefault}{\color[rgb]{0,0,0}$x\wedge x'$}%
}}}}
\end{picture}%

\caption{$\mathbf{F}_{\M}(1)$.}\label{Fig:FM1}
\end{figure}
\end{example}

Let $\V$ be a variety of algebras. An algebra $\mathbf{A} \in \V$ is said to be \emph{projective} if for every
pair of algebras $\mathbf{B},\mathbf{C}\in \V$,
every surjective homomorphism $f \colon \mathbf{B} \to\mathbf{C}$,
and every homomorphism $h \colon \mathbf{A} \to\mathbf{C}$,
there exists a homomorphism $g \colon \mathbf{A} \to\mathbf{B}$ such that $f \circ g=h$.

We exploit the following characterization of projective algebras \cite{H63}. 
%{\bf (L to S: Could we be more precise here? I don't have the book, so could you chek it?)}

\begin{theorem}\label{th:projcharOp}
Let $\V$ be a variety, and let $\mathbf{A} \in \V$.  Then,
$\mathbf{A}$ is projective in $\V$ iff
$\mathbf{A}$ is a \emph{retract} of a free algebra $\mathbf{F}_{\V}(\kappa)$ in $\V$ for some cardinal $\kappa$, 
that is, there exist homomorphisms $r \colon \mathbf{F}_{\V}(\kappa) \to \mathbf{A}$ 
and $f \colon \mathbf{A} \to \mathbf{F}_{\V}(\kappa)$ such that $r \circ f=\mathrm{id}_{\mathbf{A}}$.
\end{theorem}

\subsection{Finite Duality} \label{subsect:dualities}

We recall duality theorems for the categories 
of finite bounded distributive lattices, $\FD$, finite De Morgan algebras, $\FM$, 
and finite Kleene algebras, $\FK$.  

First, we present Birkhoff duality between finite bounded distributive 
lattices and finite posets \cite{B37}.  
The category $\FP$ of finite posets has finite posets $(P,\leq)$ as objects, and monotone maps as morphisms.  

Define the map $\JL \colon \FD \to \FP$ as follows:  For every $\mathbf{A}$ in $\FD$, let
$$
\JL(\mathbf{A})=(\{ x \mid \text{$x$ join irreducible in $\mathbf{A}$}\},\leq)\text{,}
$$
where $\leq$ is the order inherited from the order in $\mathbf{A}$.
For every $h \colon \mathbf{A} \to \mathbf{B}$ in $\FD$, let
$\JL(h)\colon\JL(\mathbf{B}) \to \JL(\mathbf{A})$, be the map defined by
$$
\JL(h)(x)=\bigwedge \{ y \mid h(y) \geq x \}\text{,}
$$
for all $x\in \JL(\mathbf{B})$.

Define the map $\DL \colon \FP \to \FD$ as follows:  For every $\mathbf{P}=(P,\leq) \in \FP$, let
$$
\DL(\mathbf{P})=(\{X \subseteq P \mid (X]=X \},\cap,\cup,\emptyset,P)\text{.}
$$
%where $[X)=\{y\in P\mid y\leq x\mbox{ for some }x\in X \}$.
For every $f \colon \mathbf{P} \to \mathbf{Q}$ in $\FP$, let
$\DL(f)\colon \DL(\mathbf{Q}) \to \DL(\mathbf{P})$ be the map defined by
$$
\DL(f)(X)=f^{-1}(X)
$$
for all $X\in \DL(\mathbf{Q})$. 
% It is easy to check that $\JL$ and $\DL$ are contravariant functors.

%\begin{theorem}[Birkhoff, \cite{B37}]\label{th:birkhoff}
%$\FD$ and $\FP$ are dually equivalent via the pair of contravariant functors $\JL$ and $\DL$.
%\end{theorem}
\begin{theorem}[Birkhoff, \cite{B37}]\label{th:birkhoff}
$\JL$ and $\DL$ are well defined contravariant functors. Moreover, they determine a dual equivalence between $\FD$ and $\FP$.
\end{theorem}

Building on the duality for bounded distributive lattices developed by Priestley \cite{P70},
in \cite[Theorem~2.3 and Theorem~3.2]{CF77} Cornish and Fowler present
a duality for De Morgan and Kleene algebras.  We rely on Theorem~\ref{th:birkhoff}, 
to restrict such duality to finite objects.

\begin{definition}[Finite Involutive Posets, $\FPM$ and $\FPK$]
The category $\FPM$ of \emph{finite involutive posets} is defined as follows:
\begin{description}
\item[$\mathrm{Objects:}$] Structures $(P,\leq,i)$, where $(P,\leq)$ is a finite poset, 
and $i \colon P \to P$ is such that $x \leq y$ implies $i(y) \leq i(x)$ and $i(i(x))=x$.
\item[$\mathrm{Morphisms:}$] Maps $f\colon (P,\leq,i)\to (P',\leq',i')$ such that 
$x \leq y$ implies $f(x) \leq' f(y)$ and $f(i(x))=i'(f(x))$.
\end{description}
The category $\FPK$ is the full subcategory of $\FPM$ whose objects 
%are 
%finite involutive posets 
$(P,\leq,i)$ are such that
$i(x)$ is comparable to $x$ for all $x \in P$.
\end{definition}

The map $\JM \colon \FM \to \FPM$ is defined by: For every $\mathbf{A}=(A,\wedge,\vee,',0,1)$ in $\FM$, let
$$
\JM(\mathbf{A})=(\JL (A,\wedge,\vee,0,1),i),
$$
where $i(x)=\bigwedge (A\setminus \{a'\mid a\in [x)\})$ for each $x\in\JL (A,\wedge,\vee,0,1)$.
Moreover, $\JM(h)=\JL(h)$ for every $h \colon \mathbf{A} \to \mathbf{B}$ in $\mathcal{FM}$.

The map $\DM \colon \FPM \to \FM$ is defined by: For every $\mathbf{P}=(P,\leq,i) \in \FPM$,
$$
\DM(\mathbf{P})=\mathbf{A}=(A,\wedge,\vee,',0,1)
$$
where $(A,\wedge,\vee,0,1)=\DL(P ,\leq)$, and $X'=P \setminus i(X)$.
Moreover, $\DM(f)=\DL(f)$ for every $f \colon \mathbf{P} \to \mathbf{Q}$ in $\FPM$.

We will denote $\JK\colon \FK\to\FPK$ and $\DK\colon \FPK\to \FK$ the restrictions of the functors $\JM$ and $\DM$ 
to the categories $\FK$ and $\FPK$ respectively. 

\begin{theorem}[Cornish and Fowler]\label{th:dualities}
$\JM$ and $\DM$ (respectively, $\JK$ and $\DK$) are well defined contravariant functors. Moreover, they determine a dual equivalence between the categories $\FM$ and $\FPM$  (respectively, $\FK$ and $\FPK$) .
\end{theorem}

Let $\D=(D,\leq,i) \in \FPM$ be as in Figure~\ref{Fig:JMFM1}.  In light of Example~\ref{example:freedm1}, 
 $\JM(\mathbf{F}_{\M}(1))$ and $\mathbf{D}$ are isomorphic 
via the map $x\wedge x' \mapsto 2$, $x \mapsto 0$, $x' \mapsto 1$, $1 \mapsto 3$.
\begin{figure}[h]
\centering
\begin{picture}(0,0)%
\includegraphics{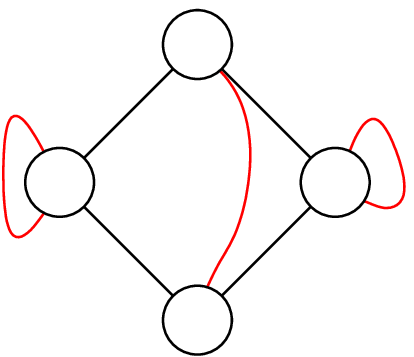}%
\end{picture}%
\setlength{\unitlength}{2901sp}%
\begingroup\makeatletter\ifx\SetFigFont\undefined%
\gdef\SetFigFont#1#2#3#4#5{%
  \reset@font\fontsize{#1}{#2pt}%
  \fontfamily{#3}\fontseries{#4}\fontshape{#5}%
  \selectfont}%
\fi\endgroup%
\begin{picture}(2664,2280)(14911,-4801)
\put(17114,-3703){\makebox(0,0)[b]{\smash{{\SetFigFont{6}{7.2}{\familydefault}{\mddefault}{\updefault}{\color[rgb]{0,0,0}$1$}%
}}}}
\put(15297,-3697){\makebox(0,0)[b]{\smash{{\SetFigFont{6}{7.2}{\familydefault}{\mddefault}{\updefault}{\color[rgb]{0,0,0}$0$}%
}}}}
\put(16200,-4607){\makebox(0,0)[b]{\smash{{\SetFigFont{6}{7.2}{\familydefault}{\mddefault}{\updefault}{\color[rgb]{0,0,0}$2$}%
}}}}
\put(16200,-2797){\makebox(0,0)[b]{\smash{{\SetFigFont{6}{7.2}{\familydefault}{\mddefault}{\updefault}{\color[rgb]{0,0,0}$3$}%
}}}}
\end{picture}%

\caption{$\JM(\mathbf{F}_{\M}(1)) \simeq \D$.  Curved edges depict the map $i \colon D \to D$.}\label{Fig:JMFM1}
\end{figure}

Let $\mathbf{P}=(P,\leq,i) \in \FPM$.  By \cite[Theorem~2.4]{CF77},
the \emph{product} of $n$ copies of $\mathbf{P}$ in the category $\FPM$,
denoted by $$\mathbf{P}^{n}=(P^n,\leq^n,i^n)\text{,}$$ is the finite 
poset over $P^n$ with the order and the involution defined coordinatewise, that is for all $x=(x_1,\dots,x_n), y=(y_1,\dots,y_n)\in P^n$,
$x \leq^n y$ iff $x_i \leq y_i$ for all $i=1,\dots,n$,
and $i^n(x)=(i(x_1),\dots,i(x_n))$.

\begin{proposition}\label{prop:freedm}
$\JM(\mathbf{F}_{\M}(n))\simeq\D^{n}$.
\begin{proof}
 $\mathbf{F}_{\M}(n)$ is %(isomorphic to)
the coproduct of $n$ copies of $\mathbf{F}_{\M}(1)$.  Therefore, by Theorem~\ref{th:dualities}, 
$\JM(\mathbf{F}_{\M}(n))\simeq\JM(\mathbf{F}_{\M}(1))^{n}\simeq\D^{n}$, the product of $n$ copies of $\JM(\mathbf{F}_{\M}(1))$.   
The statement follows.
\end{proof}
\end{proposition}

\begin{figure}[h]
\centering
\begin{picture}(0,0)%
\includegraphics{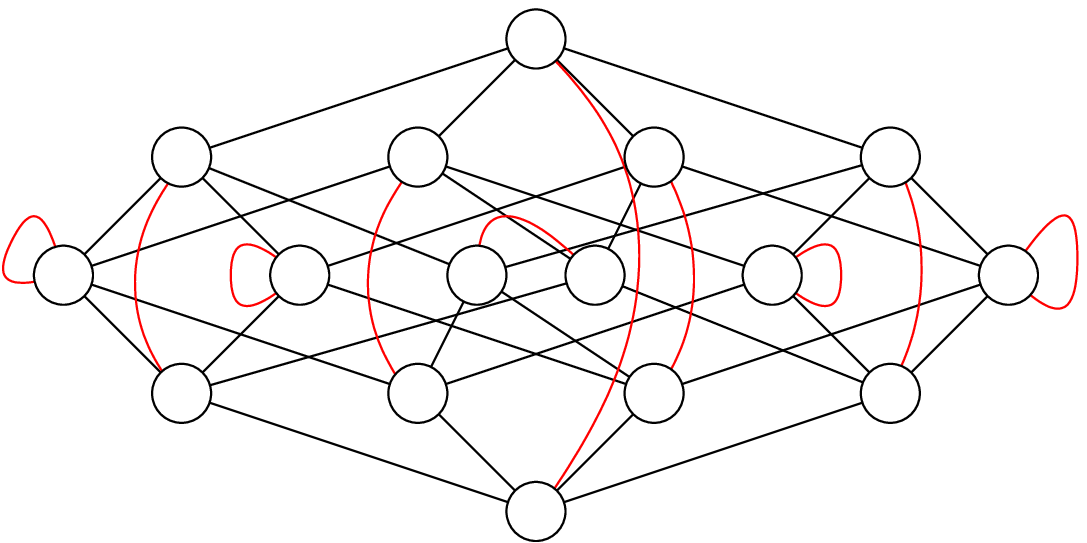}%
\end{picture}%
\setlength{\unitlength}{2486sp}%
\begingroup\makeatletter\ifx\SetFigFont\undefined%
\gdef\SetFigFont#1#2#3#4#5{%
  \reset@font\fontsize{#1}{#2pt}%
  \fontfamily{#3}\fontseries{#4}\fontshape{#5}%
  \selectfont}%
\fi\endgroup%
\begin{picture}(8229,4080)(12119,-4801)
\put(16067,-4614){\makebox(0,0)[lb]{\smash{{\SetFigFont{6}{7.2}{\familydefault}{\mddefault}{\updefault}{\color[rgb]{0,0,0}$22$}%
}}}}
\put(13384,-3720){\makebox(0,0)[lb]{\smash{{\SetFigFont{6}{7.2}{\familydefault}{\mddefault}{\updefault}{\color[rgb]{0,0,0}$02$}%
}}}}
\put(15185,-3711){\makebox(0,0)[lb]{\smash{{\SetFigFont{6}{7.2}{\familydefault}{\mddefault}{\updefault}{\color[rgb]{0,0,0}$20$}%
}}}}
\put(16990,-3708){\makebox(0,0)[lb]{\smash{{\SetFigFont{6}{7.2}{\familydefault}{\mddefault}{\updefault}{\color[rgb]{0,0,0}$21$}%
}}}}
\put(18791,-3698){\makebox(0,0)[lb]{\smash{{\SetFigFont{6}{7.2}{\familydefault}{\mddefault}{\updefault}{\color[rgb]{0,0,0}$12$}%
}}}}
\put(15651,-2819){\makebox(0,0)[lb]{\smash{{\SetFigFont{6}{7.2}{\familydefault}{\mddefault}{\updefault}{\color[rgb]{0,0,0}$23$}%
}}}}
\put(16556,-2819){\makebox(0,0)[lb]{\smash{{\SetFigFont{6}{7.2}{\familydefault}{\mddefault}{\updefault}{\color[rgb]{0,0,0}$32$}%
}}}}
\put(17900,-2823){\makebox(0,0)[lb]{\smash{{\SetFigFont{6}{7.2}{\familydefault}{\mddefault}{\updefault}{\color[rgb]{0,0,0}$10$}%
}}}}
\put(14310,-2823){\makebox(0,0)[lb]{\smash{{\SetFigFont{6}{7.2}{\familydefault}{\mddefault}{\updefault}{\color[rgb]{0,0,0}$01$}%
}}}}
\put(15203,-1914){\makebox(0,0)[lb]{\smash{{\SetFigFont{6}{7.2}{\familydefault}{\mddefault}{\updefault}{\color[rgb]{0,0,0}$30$}%
}}}}
\put(17000,-1910){\makebox(0,0)[lb]{\smash{{\SetFigFont{6}{7.2}{\familydefault}{\mddefault}{\updefault}{\color[rgb]{0,0,0}$31$}%
}}}}
\put(16104,-996){\makebox(0,0)[lb]{\smash{{\SetFigFont{6}{7.2}{\familydefault}{\mddefault}{\updefault}{\color[rgb]{0,0,0}$33$}%
}}}}
\put(18807,-1911){\makebox(0,0)[lb]{\smash{{\SetFigFont{6}{7.2}{\familydefault}{\mddefault}{\updefault}{\color[rgb]{0,0,0}$13$}%
}}}}
\put(19703,-2807){\makebox(0,0)[lb]{\smash{{\SetFigFont{6}{7.2}{\familydefault}{\mddefault}{\updefault}{\color[rgb]{0,0,0}$11$}%
}}}}
\put(12504,-2810){\makebox(0,0)[lb]{\smash{{\SetFigFont{6}{7.2}{\familydefault}{\mddefault}{\updefault}{\color[rgb]{0,0,0}$00$}%
}}}}
\put(13405,-1905){\makebox(0,0)[lb]{\smash{{\SetFigFont{6}{7.2}{\familydefault}{\mddefault}{\updefault}{\color[rgb]{0,0,0}$03$}%
}}}}
\end{picture}%

\caption{$\JM(\mathbf{F}_{\M}(2)) \simeq \D^{2}$.}\label{Fig:JMFM2}
\end{figure}

Let $\mathbf{P}=(P,\leq,i) \in \FPM$.  By \cite{A87},
subobjects of $\mathbf{P}$ are subsets $X \subseteq P$ with the inherited order such that $X=i(X)$.
By Theorem~\ref{th:dualities}, subobjects of $\mathbf{P}$ correspond exactly to quotients on $\DM(\mathbf{P})$.  For each $\mathbf{P}$ in $\FPM$, 
let $\mathbf{P}_k$ be the largest subobject of $\mathbf{P}$ lying in the subcategory $\FPK$, that is,
$\mathbf{P}_k$ is the subobject of $\mathbf{P}$ (possibly empty) such that each element $x$ of $\mathbf{P}_k$ is comparable with $i(x)$.
Therefore, $\DM(\mathbf{P}_k)$ is the largest quotient of $\DM(\mathbf{P})$ lying in $\FK$.

\begin{proposition}\label{prop:freek}
$\JK(\mathbf{F}_{\K}(n))\simeq (\D^{n})_k$.
\begin{proof}
$\mathbf{F}_{\K}(n)$ is the largest quotient of $\mathbf{F}_{\M}(n)$ that is a Kleene algebra.
By Proposition~\ref{prop:freedm}, $\JM(\mathbf{F}_{\M}(n)) \simeq \D^{n}$. Then
by the mentioned correspondence between quotients and subobjects under the duality \cite{A87},
$\JK(\mathbf{F}_{\K}(n))=\JM(\mathbf{F}_{\K}(n))$ arises as the largest subobject of $\D^{n}$ lying in $\FPK$.
\end{proof}
\end{proposition}

\begin{figure}[h]
\centering
\begin{picture}(0,0)%
\includegraphics{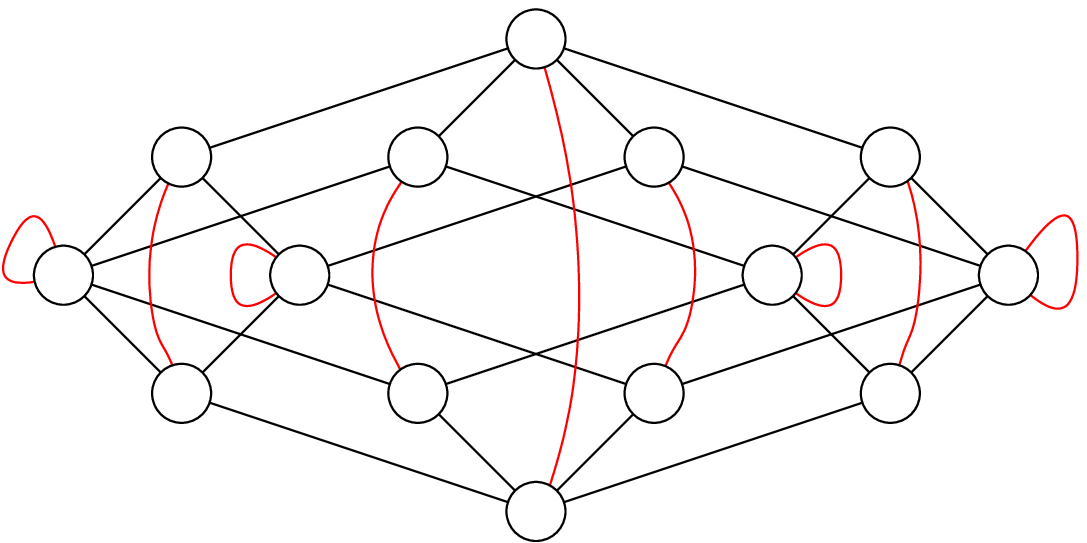}%
\end{picture}%
\setlength{\unitlength}{2486sp}%
\begingroup\makeatletter\ifx\SetFigFont\undefined%
\gdef\SetFigFont#1#2#3#4#5{%
  \reset@font\fontsize{#1}{#2pt}%
  \fontfamily{#3}\fontseries{#4}\fontshape{#5}%
  \selectfont}%
\fi\endgroup%
\begin{picture}(8229,4080)(12119,-4801)
\put(16067,-4614){\makebox(0,0)[lb]{\smash{{\SetFigFont{6}{7.2}{\familydefault}{\mddefault}{\updefault}{\color[rgb]{0,0,0}$22$}%
}}}}
\put(13384,-3720){\makebox(0,0)[lb]{\smash{{\SetFigFont{6}{7.2}{\familydefault}{\mddefault}{\updefault}{\color[rgb]{0,0,0}$02$}%
}}}}
\put(15185,-3711){\makebox(0,0)[lb]{\smash{{\SetFigFont{6}{7.2}{\familydefault}{\mddefault}{\updefault}{\color[rgb]{0,0,0}$20$}%
}}}}
\put(16990,-3708){\makebox(0,0)[lb]{\smash{{\SetFigFont{6}{7.2}{\familydefault}{\mddefault}{\updefault}{\color[rgb]{0,0,0}$21$}%
}}}}
\put(18791,-3698){\makebox(0,0)[lb]{\smash{{\SetFigFont{6}{7.2}{\familydefault}{\mddefault}{\updefault}{\color[rgb]{0,0,0}$12$}%
}}}}
\put(17900,-2823){\makebox(0,0)[lb]{\smash{{\SetFigFont{6}{7.2}{\familydefault}{\mddefault}{\updefault}{\color[rgb]{0,0,0}$10$}%
}}}}
\put(14310,-2823){\makebox(0,0)[lb]{\smash{{\SetFigFont{6}{7.2}{\familydefault}{\mddefault}{\updefault}{\color[rgb]{0,0,0}$01$}%
}}}}
\put(15203,-1914){\makebox(0,0)[lb]{\smash{{\SetFigFont{6}{7.2}{\familydefault}{\mddefault}{\updefault}{\color[rgb]{0,0,0}$30$}%
}}}}
\put(17000,-1910){\makebox(0,0)[lb]{\smash{{\SetFigFont{6}{7.2}{\familydefault}{\mddefault}{\updefault}{\color[rgb]{0,0,0}$31$}%
}}}}
\put(16104,-996){\makebox(0,0)[lb]{\smash{{\SetFigFont{6}{7.2}{\familydefault}{\mddefault}{\updefault}{\color[rgb]{0,0,0}$33$}%
}}}}
\put(18807,-1911){\makebox(0,0)[lb]{\smash{{\SetFigFont{6}{7.2}{\familydefault}{\mddefault}{\updefault}{\color[rgb]{0,0,0}$13$}%
}}}}
\put(19703,-2807){\makebox(0,0)[lb]{\smash{{\SetFigFont{6}{7.2}{\familydefault}{\mddefault}{\updefault}{\color[rgb]{0,0,0}$11$}%
}}}}
\put(12504,-2810){\makebox(0,0)[lb]{\smash{{\SetFigFont{6}{7.2}{\familydefault}{\mddefault}{\updefault}{\color[rgb]{0,0,0}$00$}%
}}}}
\put(13405,-1905){\makebox(0,0)[lb]{\smash{{\SetFigFont{6}{7.2}{\familydefault}{\mddefault}{\updefault}{\color[rgb]{0,0,0}$03$}%
}}}}
\end{picture}%

\caption{$\JK(\mathbf{F}_{\K}(2))\simeq(\D^{2})_k$.}\label{Fig:JMFK2}
\end{figure}

\subsection{Unification Theory}

%For background on unification theory, we refer the reader to \cite{BS01}.  
Let $\mathbf{P}=(P,\leq)$ be a preorder.  A \emph{$\mu$-set} for $\mathbf{P}$
is a subset $M \subseteq P$ such that $x \parallel y$ for all $x,y \in M$ such that $x\neq y$,
and for every $x \in P$ there exists $y \in M$ such that $x \leq y$.  It is easy to check that if $\mathbf{P}$ has a $\mu$-set, 
then every $\mu$-set of $\mathbf{P}$ has the same cardinality.

We say that $\mathbf{P}$ has type:

\begin{itemize}
\item[] \emph{nullary} if $\mathbf{P}$ has no $\mu$-sets (in symbols, $\mathrm{type}(\mathbf{P})=0$);
\item[] \emph{infinitary} if $\mathbf{P}$ has a $\mu$-set of infinite cardinality ($\mathrm{type}(\mathbf{P})=\infty$);
\item[] \emph{finitary} if $\mathbf{P}$ has a finite $\mu$-set of cardinality greater than $1$ ($\mathrm{type}(\mathbf{P})=\omega$);
\item[] \emph{unitary} if $\mathbf{P}$ has a $\mu$-set of cardinality $1$ ($\mathrm{type}(\mathbf{P})=1$).
\end{itemize}

We prepare for later use some easy consequences of the definitions.

\begin{lemma}\label{Lemma:UpsetNullary}
The set $\{0,\omega,\infty,0\}$ carries a natural total order $1\leq \omega \leq \infty \leq 0$. If $\mathbf{P}$ is a preorder and $Q\subseteq P$ be an upset of $\mathbf{P}$ and $\mathbf{Q}$ denotes the preorder with universe $Q$ and relation inherited from $\mathbf{P}$,  then
${\rm type}(\mathbf{Q})\leq{\rm type}(\mathbf{P})$. 
\end{lemma}

\begin{lemma}\label{Lemma:DirectedType}
Let $\mathbf{P}=(P,\leq)$ be a directed preorder. Then, ${\rm type}(\mathbf{P})=0$ or ${\rm type}(\mathbf{P})=1$.
\end{lemma}

The algebraic unification theory by Ghilardi \cite{G97}
reduces the traditional symbolic unification problem over an equational theory to the following:

\begin{description}
\item[Problem]  $\textsc{Unif}(\mathcal{V})$.
\item[Instance]  A finitely presented algebra $\mathbf{A} \in \mathcal{V}$.
\item[Solution]  A homomorphism $u \colon \mathbf{A} \to \mathbf{P}$, where $\mathbf{P}$ is a finitely presented projective algebra in $\mathcal{V}$.
\end{description}

A solution to an instance $\A$ is called an \emph{(algebraic) unifier} for $\A$, 
and $\A$ is called \emph{solvable in $\V$} if $\A$ has a solution.

Let $\A \in \V$ be finitely presented,
and for $i=1,2$ let $u_i \colon \A \to \mathbf{P}_i$ be a unifier for $\A$.
Then, $u_1$ is more general than $u_2$, in symbols, $u_2 \leq u_1$,
if there exists a homomorphism $f \colon \mathbf{P}_1 \to \mathbf{P}_2$
such that $f \circ u_1=u_2$. For $\A$ solvable in $\V$, let $U_{\V}(\A)$ be the preorder induced by the generality relation
over the unifiers for $\A$.  We define the type of $\A$ as the type of the preordered set $U_{\V}(\A)$, in symbols
$\mathrm{type}_{\V}(\A)=\mathrm{type}(U_{\V}(\A))$.  

We say that the variety $\V$ has type:
\begin{itemize}
\item[] \emph{nullary} 
if $\{ \mathrm{type}_{\V}(\mathbf{A}) \mid \text{$\mathbf{A}$ solvable in $\V$} \} \cap \{0\} \neq \emptyset$;
\item[] \emph{infinitary} 
if $\infty\in \{ \mathrm{type}_{\V}(\mathbf{A}) \mid \text{$\mathbf{A}$ solvable in $\V$} \}\subseteq \{\infty,\omega,1\}$;
\item[] \emph{finitary} 
if $\omega\in\{ \mathrm{type}_{\V}(\mathbf{A}) \mid \text{$\mathbf{A}$ solvable in $\V$} \}\subseteq \{\omega,1\}$;
\item[] \emph{unitary} 
if $\{ \mathrm{type}_{\V}(\mathbf{A}) \mid \text{$\mathbf{A}$ solvable in $\V$} \}=\{1\}$.
\end{itemize} 

\section{Finite Projective Algebras}\label{sect:fpdmk}

We provide first-order decidable characterizations of the finite involutive 
posets corresponding to finite %(equivalently, finitely presented, or finitely generated) 
projective De Morgan (Theorem~\ref{th:mmain}) 
and Kleene (Theorem~\ref{th:kmain}) algebras.

%under the dualities introduced in Section~\ref{subsect:dualities}.  

%We collect from \cite{BB89} the following notion.

\begin{definition}[\cite{BB89}]\label{def:kcompl}
Let $\kappa$ be a cardinal.  A poset $(P,\leq)$ is $\kappa$-complete 
if, whenever $X \subseteq P$ is such that all $Y \subseteq X$ with $|Y|<\kappa$ have an upper bound, 
then $\bigvee X$ exists in $(P,\leq)$.
\end{definition}

\begin{theorem}[De Morgan Projective]\label{th:mmain}
Let $\mathbf{A} \in \FM$.  Then $\mathbf{A}$ is projective in $\M$
iff $\JM(\mathbf{A})=(P,\leq,i) \in \FPM$
satisfies the following:
\begin{enumerate}
\item[$(M_1)$] $(P,\leq)$ is a nonempty lattice;
\item[$(M_2)$] for all $x \in P$, if $x \leq i(x)$, then there exists $y\in P$ such that $x\leq y=i(y)$;
\item[$(M_3)$] $\{x\in P\mid x\leq i(x)\}$ with inherited order is $3$-complete.
\end{enumerate}
\end{theorem}
\begin{proof}
%Let $\mathbf{A} \in \FM$ and let $\JM(\mathbf{A})=\mathbf{P}=(P,\leq,i) \in \FPM$.
There exists $n \in \mathbb{N}$ such that $\JM(\mathbf{A})=(P,\leq,i)=\mathbf{P}$ 
is a subobject of $\JM(\mathbf{F}_{\M}(n))=\mathbf{D}^n=(D^n,\leq,i) \in \FPM$ by Proposition~\ref{prop:freedm}.
That is, it is possible to display $\mathbf{P}$ as a subset of $\mathbf{D}^n$ with inherited order and involution.
Combining this together with Theorem~\ref{th:projcharOp} and Theorem \ref{th:dualities}, $\mathbf{A}$ is projective 
iff there exists an onto morphism $r\colon \mathbf{D}^n\rightarrow \mathbf{P}$ in $\FPM$ such that $r\circ r=r$, that is $r|_P=\mathrm{id}_{\mathbf{P}}$, where $r|_P$ denotes the restriction of $r$ to $P$.  
Therefore, it is sufficient to show that conditions $(M_1)$-$(M_3)$ are necessary and sufficient for the existence of such a map $r$.

Below, $Z=\{ z \in D^n \mid z=i(z) \}=\{0,1\}^{n}$ and $Y=Z \cap P$.%$Y=\{ z \in P \mid z=i(z) \}=Z \cap P$.

\medskip

$(\Rightarrow)$ Let $r\colon \mathbf{D}^n\rightarrow \mathbf{P}$ be a morphism in $\FPM$ such that $r|_P=\mathrm{id}_{\mathbf{P}}$.  
We show that $\mathbf{P}$ satisfies $(M_1)$, $(M_2)$, and $(M_3)$.

For $(M_1)$:  In particular, $r$ is a poset retraction of $\mathbf{D}^n$ onto $\mathbf{P}$.  
Since $\mathbf{D}^n$ is a nonempty lattice, it follows straightforwardly  that $\mathbf{P}$ is a nonempty lattice.  

For $(M_2)$: Let $y \in P$ be such that $y \leq i(y)$. Then $y\in\{2,0,1\}^{n}$.
Let $z \in Z$ be such that for $i=1,\dots,n$, if $y_i \in \{0,1\}$ then $z_i=y_i$ and $z_i=0$ otherwise.
Then $y \leq z \leq i(y)$, and $y=r(y) \leq r(z)=r(i(z))=i(r(z))$.

For $(M_3)$: Observe that $r((Z])=(Y]_P$, because if $x \leq z$ for $x \in D^n$ and $z \in Z$,
then $r(x) \leq r(z) \in Y$.  Then the restriction of
$r$ to $(Z]$ is a poset retraction of $(Z]$ onto $(Y]_P$. Since $(Z]=\{2,0,1\}^n$ is $3$-complete, 
by \cite[Corollary~2.6]{BB89} $r((Z])=(Y]_P$ is $3$-complete.
% GIVE IT? Let $W \subseteq (Y]$ be such that $x \vee y \in (Y]$ for all $x,y \in (Y]$.
% As $(P,\leq)$ is a lattice, $\bigvee^P W$ exists in $(P,\leq)$.
% We show $\bigvee^P W \in (Y]$.  Letting $W_i \leftrightharpoons \{ x_i \mid x \in W \}$,
% there does not exist $i \in \{1,\dots,n\}$ such that $\{0,1\} \subseteq W_i$,
% otherwise if say $x_i=0$ and $y_i=1$ for $x,y \in W$, then $(x \vee y)_i=3$ and $x \vee y \not\in (Y]$.
% Thus, $\bigvee^n W \in (Z]$.  Observe that $r((Z])=(Y]$,
% because if $x \leq z$ for some $z \in Z$,
% then $r(x) \leq r(z) \in Y$ by the properties of $r$.
% Then $r(\bigvee^n W)$ is in $(Y]$ and is an upper bound of $W$ by monotonicity,
% so $\bigvee^P W \in (Y]$.

\medskip

$(\Leftarrow)$ Assume that $\mathbf{P}$ satisfies $(M_1)$, $(M_2)$, and $(M_3)$.  
We show that there exists a morphism $r\colon \mathbf{D}^n\rightarrow \mathbf{P}$ in $\FPM$ such that $r|_P=\mathrm{id}_{\mathbf{P}}$.  

To define the retraction $r \colon D^n \to P$, we first introduce the following notation. For all $x \in D^n$,
let $L_x = \{ z \in P \mid z \leq x \}$ and $U_x =\{ z \in P \mid x \leq z \}$. Since $i$ is an order reversing involution, 
\begin{equation}\label{Eq:LxUx}
i(L_x)=U_{i(x)}\text{,}
\end{equation}
for each $x\in D^{n}$.

If $x \in Z$, then there exists $y \in Y$ is such that $\bigvee_P L_x \leq y\leq \bigwedge_P U_x$. In fact, 
if $z_1,z_2\in L_x$, then $z_1,z_2\leq x\leq i(z_1),i(z_2)$. By $(M_1)$, $z_1\vee_{P}z_2\leq i(z_1)\wedge_{P}i(z_2)=i(z_1\vee_{P}z_2)$.  
Combining this with $(M_3)$, we have $\bigvee_{P}L_x\leq i(\bigvee_{P}L_x)$.  Now, by $(M_2)$, there exists $y\in Y$ such that 
$\bigvee_{P}L_x\leq y$. Finally by \eqref{Eq:LxUx} and the fact that $x=i(x)$, 
$\bigvee_P L_x \leq y=i(y)\leq i(\bigvee_P L_x )=\bigwedge_P i(L_{x})=\bigwedge_P U_{i(x)}=\bigwedge_P U_{x}$, as desired.

For each $x \in Z$, we fix $r(x)\in Y$ such that
\begin{equation}\label{eq:mstretr1}
\textstyle \bigvee_P L_x \leq r(x)\leq \bigwedge_P U_x\text{.}
\end{equation}
If $x\in D^{n}\setminus Z$,
then let $m$ be the smallest number in $\{1,\dots,n\}$ such that $x_m \in \{2,3\}$.  We define,
\begin{align}\label{eq:mstretr2}
r(x) &=
\begin{cases}
\bigvee_P L_x\text{,} & \text{if $x_m=2$;}\\
\bigwedge_P U_x\text{,} & \text{if $x_m=3$.}
\end{cases}
\end{align}

The map $r\colon D^{n}\rightarrow P$ is well defined.
Also, $r(x)=x$ for all $x \in P$ because $x\in U_x\cap L_x$.  %Thus, $r$ is a retraction.

\medskip

{\it  Claim 1:} For each $x \in D^n$, $r(i(x))=i(r(x))$.

If $x \in Z$, then $r(i(x))=i(r(x))$ holds because $r(x) \in Y$.
Let $x \in D^n \setminus Z$, and $m$ be the smallest number in $\{1,\dots,n\}$ such that $x_m \in \{2,3\}$.
If $x_m=2$ and $(i(x))_m=3$, then $r(x)=\bigvee_P L_x$ and $r(i(x))=\bigwedge_P U_{i(x)}$.
Then by \eqref{Eq:LxUx},
\begin{align*}
r(i(x)) &=\textstyle \bigwedge_P U_{i(x)}=\bigwedge_P i(L_{x}) \\
        &=\textstyle i\left(\bigvee_P L_x \right) \\
        &= i(r(x)) \text{.}
\end{align*}
The case $x_m=3$ and $(i(x))_m=2$ reduces to the previous case, which concludes the proof of the claim.

\medskip

{\it  Claim 2:}  $r$ is monotone.

Let $x,y\in D^{n}$ such that $x<y$. Then $L_x\subseteq L_y$, $U_y\subseteq U_x$.
% and $z\leq w$ for every $z\in L_x$ and $w\in U_y$.
Therefore,
\begin{equation}\label{Eq:PropLxLyI}
\textstyle\bigvee_{P} L_x\leq \bigvee_{P} L_y\leq \bigwedge_{P} U_y\text{,}
\end{equation}
and
\begin{equation}\label{Eq:PropLxLyII}
 \textstyle\bigwedge_{P} U_x\leq \bigwedge_{P} U_y\text{.}
\end{equation}
If $x\in Z$, observe that $y\in\{0,1,3\}^{n}\setminus\{0,1\}^{n}$. By \eqref{eq:mstretr1}, \eqref{eq:mstretr2}  and \eqref{Eq:PropLxLyII}, $r(x)\leq \bigwedge_{P} U_x\leq \bigwedge_{P} U_y=r(y)$. A similar argument proves that $r(x)\leq r(y)$ if $y\in Z$.

If $x,y \in D^n \setminus Z$. If $r(x)=\bigvee_{P} L_x$, then $r(x) \leq r(y)$ by \eqref{Eq:PropLxLyI} and \eqref{eq:mstretr2}.
If $r(x)=\bigwedge_{P} U_x$, then by \eqref{eq:mstretr2},
letting $m$ be the smallest number in $\{1,\ldots,n\}$ such that $x_m=3$ since $x\leq y$
it follows that $y_k\in\{0,1,3\}$ for every $k<m$, and $y_m=3$.
Again by \eqref{eq:mstretr2}, $r(y)=\bigwedge_{P} U_y$. Therefore, $r(x)\leq r(y)$ by \eqref{Eq:PropLxLyII}, which concludes the proof of the claim.

\medskip

By Claim 1 and Claim 2, $r \colon \mathbf{D}^{n} \to \mathbf{P}$ is the required retraction, 
so that $\mathbf{A}$ is a retract of $\mathbf{F}_{\M}(n)$ in $\FM$, that is, it is projective in $\FM$.
\end{proof}

Since $(P,\leq)$ is a finite lattice by $(M_1)$, condition $(M_3)$ reduces to the following first-order statement:
$x \vee y \vee z\leq i(x \vee y \vee z)$, for all $x,y,z$ such that $x \vee y\leq i(x\vee y)$, $x \vee z\leq i(x\vee z)$, and $y \vee z \leq i(y\vee z)$.

\begin{theorem}[Kleene Projective]\label{th:kmain}
Let $\mathbf{A} \in \FK$.  Then $\mathbf{A}$ is projective in $\K$
iff $\JK(\mathbf{A})=(P,\leq,i) \in \FPK$ satisfies conditions $(M_2)$, $(M_3)$ in Theorem~\ref{th:mmain} and the conditions:
\begin{enumerate}
\item[$(K_1)$] $\{x\in P\mid x\leq i(x)\}$ with inherited order is a nonempty meet semilattice;
\item[$(K_2)$] every $x,y \in P$ such that $x ,y \leq i(y),i(x)$ have a common upper bound $z \in P$ such that $z \leq i(z)$.
\end{enumerate}
\end{theorem}

\begin{proof}
There exists $n \in \mathbb{N}$ such that $\JK(\mathbf{A})=(P,\leq,i)=\mathbf{P}$ 
is a subobject of $\JK(\mathbf{F}_{\K}(n))=(\mathbf{D}^n)_k=((D^n)_k,\leq,i) \in \FPK$ by Proposition~\ref{prop:freek}.
Combining Theorem~\ref{th:projcharOp} and Theorem \ref{th:dualities}, $\mathbf{A}$ is projective 
iff there exists a morphism $r\colon (\mathbf{D}^n)_k \rightarrow \mathbf{P}$ in $\FPK$ such that $r|_P=\mathrm{id}_{\mathbf{P}}$.  
Therefore, it is sufficient to show that conditions $(K_1)$, $(K_2)$, $(M_2)$, and $(M_3)$ are necessary and sufficient for the existence of such a map $r$.

Below, $Z=\{ z \in (D^n)_k \mid z=i(z) \}=\{0,1\}^n$ and $Y=Z \cap P$.
\medskip

$(\Rightarrow)$ Let $r\colon (\mathbf{D}^n)_k \rightarrow \mathbf{P}$ be a morphism in $\FPK$ such that $r|_P=\mathrm{id}_{\mathbf{P}}$.

The proof that $\mathbf{P}$ satisfies $(M_2)$ and $(M_3)$
follows by the same argument used in the proof of Theorem \ref{th:mmain}.

For $(K_1)$: First observe that $r(Z)=Y$ and $r((Z])=(Y]_P$.  Then the restriction of $r$ to $(Z]$ is a poset retraction of $(Z]$ onto $(Y]_P$.  
Since $(Z]$ is a nonempty meet semilattice, 
$(Y]_P$ is a nonempty meet semilattice \cite[Lemma~2.4]{BB89}.  

For $(K_2)$: Let $x,y \in P$, be such that $x,y \leq_P i(x),i(y)$. 
Then there does not exist $i \in \{1,\dots,n\}$ such that $x_i=0$ and $y_i=1$ 
(otherwise, $x \parallel_P i(y)$), which proves that $x \vee_{D^{n}} y \in \{2,0,1\}^n=(Z]\subseteq D^{n}_k$.  
Then $z=r(x \vee_{D^{n}} y) \in (Y]_P$.  
Therefore, $x,y \leq z\leq i(z)$, as desired.

\medskip

$(\Leftarrow)$ Let $\mathbf{P}$ be a subset of $(\mathbf{D}^n)_k$ with inherited order and involution 
satisfying $(K_1)$,  $(K_2)$, $(M_2)$ and $(M_3)$.  
We define a morphism $r\colon (\mathbf{D}^n)_k \rightarrow \mathbf{P}$ in $\FPK$ such that $r|_P=\mathrm{id}_{\mathbf{P}}$.  

Since $\mathbf{P}\in\FPK$, by $(M_2)$ we have that $P=(Z]\cup[Z)=(Z]\cup i((Z])$.  Moreover, since $Z$ is an antichain, $(Z]\cap[Z)=Z$.  
For all $x \in (Z]$, let $L_x=\{ z \in P \mid z \leq x \} $.  Observe that $L_x \subseteq (Y]_P$ by $(M_2)$.  
Also, if $v,w \in L_x$, then $v,w \leq x \leq i(v),i(w)$, 
and combining $(K_1)$ and $(K_2)$, we obtain that $v \vee_P w \in (Y]_P$ for all $v,w \in L_x$.  Therefore, $\bigvee_P L_x \in (Y]_P$ by $(M_3)$.

For all $x \in Z$, we define $r(x)\in Y$ such that,
\begin{equation}\label{eq:kstretr2}
\textstyle\bigvee_P L_x\leq r(x),
\end{equation}
whose existence is ensured by condition $(M_2)$.
And for all $x \in (Z] \setminus Z$, we define,
\begin{align}\label{eq:kstretr1}
r(x) &= \textstyle\bigvee_P L_x\text{,}\\
r(i(x)) &= i(r(x))\text{.}
\end{align}

The map $r\colon (D^{n})_k\rightarrow P$ is well defined.  We prove that $r$ is the desired retraction.  
Clearly, if $x \in P$, then $r(x)=x$.
By definition, $r$ commutes with $i$.
We check monotonicity.  Let $x,y \in (D^n)_k$
be such that $x < y$.  If $x,y \in (Z]$, then $L_x \subseteq L_y$, then
$$\textstyle r(x)=\bigvee_P L_x \leq \bigvee_P L_y \leq r(y)\text{,}$$
where the last inequality always holds by (\ref{eq:kstretr2}) and (\ref{eq:kstretr1}).
If $x \in (Z]$ and $y \in [Z)$, then there exists $z \in Z$ such that $x \leq z \leq y$,
so that $i(y) \leq z$.  Then $r(x),r(i(y)) \leq r(z)$ by the previous case,
but $r(i(y))=i(r(y))$ by commutativity of $r$,
so that $r(z)=r(i(z))=i(r(z)) \leq r(y)$ by the properties of $i$ and commutativity of $r$.
If $x,y \in [Z)$, then $i(x),i(y) \in (Z]$ and $i(y) \leq i(x)$,
then $r(i(y))\leq r(i(x))$, then $i(r(y))\leq i(r(x))$,
and so $r(x)\leq r(y)$.
\end{proof}

Observe that the conditions $(K_1)$, $(K_2)$ are first-order conditions on the set $\{x\in P\mid x\leq i(x)\}$.

\section{Classification of Unification Problems}\label{sect:unif}

%In this section, using the characterization of finite projective De Morgan and Kleene algebras in Section~\ref{sect:fpdmk}, 
We obtain a complete, decidable, first-order classification of unification problems over 
bounded distributive lattices (Theorem~\ref{Thm:UnifLat}), 
Kleene algebras (Theorem~\ref{Theo:UnifClassKleene}) 
and De Morgan (Theorem~\ref{Theo:UnifClassDeMorgan}) algebras with respect to unification type.  
In particular, we establish that unification over the varieties of De Morgan and Kleene algebras is nullary.  

For the sake of presentation, we introduce the following notion.  An \emph{alphabet} $\Sigma$ is a set of letters.  A \emph{word} over $\Sigma$ 
is a finite sequence of letters in $\Sigma$.  A \emph{formal language} over $\Sigma$ 
is a subset of words over $\Sigma$.

\subsection{Distributive Lattices}\label{subsect:unifbdl}

We classify all solvable instances of the unification problem 
over bounded distributive lattices with respect to their unification type (Theorem \ref{Thm:UnifLat}), 
thus tightening the nullarity result by Ghilardi in \cite{G97}.  
This case study prepares the technically more involved cases of Kleene and De Morgan algebras.  

In \cite{BH70b}, Balbes and Horn characterize projective bounded distributive lattices.  
In the finite case, the characterization states that a finite bounded distributive lattice $\mathbf{L} \in \FD$ is projective 
iff the finite poset $\JL(\mathbf{L}) \in \FP$ is a nonempty lattice.  
Thus, combining the algebraic unification theory developed by Ghilardi \cite{G97} 
and the finite duality by Birkhoff (Theorem~\ref{th:birkhoff}), 
a unification problem over bounded distributive lattices reduces to the following combinatorial question:

\begin{description}
\item[Problem]  $\textsc{Unif}(\BDL)$.
\item[Instance]  A finite poset $\mathbf{P}$.
\item[Solution]  A monotone map $u \colon \mathbf{L} \to \mathbf{P}$,
where $\mathbf{L}$ is a finite nonempty lattice.
\end{description} 

Let $\mathbf{P}=(P,\leq) \in \FP$, and for $i=1,2$ let $u_i \colon \mathbf{L}_i \to \mathbf{P}$ be  unifiers for $\mathbf{P}$.  
Then $u_1$ is more general than $u_2$, in symbols, $u_2 \leq u_1$, 
iff there exists a monotone map $f \colon \mathbf{L}_2 \to \mathbf{L}_1$ 
such that $u_1 \circ f=u_2$.  Let $U_{\BDL}(\mathbf{P})$ denote the preordered set of unifiers of $\mathbf{P}$.  
Then, the unification type of $\mathbf{P}$ is defined as usual, 
$\mathrm{type}_{\BDL}(\mathbf{P})=\mathrm{type}(U_{\BDL}(\mathbf{P}))$.  
By the duality in Theorem~\ref{th:dualities}, 
$U_{\BDL}(\mathbf{P})$ and $U_{\BDL}(D(\mathbf{P}))$ are equivalent (as categories).  
Then, $\mathrm{type}_{\BDL}(\mathbf{P})=\mathrm{type}_{\BDL}(\DL(\mathbf{P}))$.  

\begin{remark}
An instance $\mathbf{P}=(P,\leq)$ of $\textsc{Unif}(\BDL)$
is solvable iff $P \neq \emptyset$.
\end{remark}

We now embark in the proof of the main result of this section.  The structure of the 
proof is the following:  using a slight modification of \cite[Theorem~5.7]{G97}, 
we identify a sufficient condition for an instance of the unification problem 
to have nullary type (Lemma~\ref{lem:NullBDL}), and then we prove that the identified condition is indeed necessary for nullarity 
(Theorem~\ref{Thm:UnifLat}). 

\begin{lemma}\label{lem:NullBDL}
Let $\mathbf{Q}=(Q,\leq) \in \FP$ be an instance of $\textsc{Unif}(\BDL)$.
If there exist $x,a,b,c,d,y \in Q$ such that:
\begin{itemize}
\item[$(i)$] $x\leq a,b \leq c,d \leq y$;
\item[$(ii)$] there does not exist $e\in Q$ such that $a,b \leq e \leq c,d$;
\end{itemize}
then $\mathrm{type}_{\BDL}(\mathbf{Q})=0$ (see Figure~\ref{fig:bdlconf1}).
\begin{figure}
\centering
\begin{picture}(0,0)%
\includegraphics{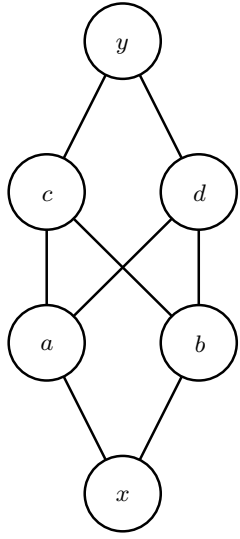}%
\end{picture}%
\setlength{\unitlength}{4144sp}%
\begingroup\makeatletter\ifx\SetFigFont\undefined%
\gdef\SetFigFont#1#2#3#4#5{%
  \reset@font\fontsize{#1}{#2pt}%
  \fontfamily{#3}\fontseries{#4}\fontshape{#5}%
  \selectfont}%
\fi\endgroup%
\begin{picture}(1380,3180)(6286,-5701)
\put(6974,-2806){\makebox(0,0)[b]{\smash{{\SetFigFont{9}{10.8}{\familydefault}{\mddefault}{\updefault}{\color[rgb]{0,0,0}$y$}%
}}}}
\put(6977,-5508){\makebox(0,0)[b]{\smash{{\SetFigFont{9}{10.8}{\familydefault}{\mddefault}{\updefault}{\color[rgb]{0,0,0}$x$}%
}}}}
\put(7430,-3714){\makebox(0,0)[b]{\smash{{\SetFigFont{9}{10.8}{\familydefault}{\mddefault}{\updefault}{\color[rgb]{0,0,0}$d$}%
}}}}
\put(6530,-3713){\makebox(0,0)[b]{\smash{{\SetFigFont{9}{10.8}{\familydefault}{\mddefault}{\updefault}{\color[rgb]{0,0,0}$c$}%
}}}}
\put(7434,-4613){\makebox(0,0)[b]{\smash{{\SetFigFont{9}{10.8}{\familydefault}{\mddefault}{\updefault}{\color[rgb]{0,0,0}$b$}%
}}}}
\put(6529,-4604){\makebox(0,0)[b]{\smash{{\SetFigFont{9}{10.8}{\familydefault}{\mddefault}{\updefault}{\color[rgb]{0,0,0}$a$}%
}}}}
\end{picture}%

\caption{Subposet of $\mathbf{Q}$ in Lemma~\ref{lem:NullBDL}.}\label{fig:bdlconf1}
\end{figure}
\end{lemma}
\begin{proof}
Since $\mathbf{Q}$ is a finite poset, we assume without loss of generality $x \in \mathrm{min}(\mathbf{Q})$ 
and $y \in \mathrm{max}(\mathbf{Q})$.  By $(ii)$, we have $a\neq b$ and $c\neq d$.  Let,
$$V=\{ u \colon \mathbf{P} \rightarrow \mathbf{Q} \in U_{\BDL}(\mathbf{Q}) \mid x,y \in u(P)\}\text{.}$$

Since $V$ is an upset of $U_{\BDL}(\mathbf{Q})$, by Lemma \ref{Lemma:UpsetNullary}, it is enough to prove that ${\rm type}(V)=0$ 
to conclude that ${\rm type}(U_{\BDL}(\mathbf{Q}))={\rm type}_{\BDL}(\mathbf{Q})=0$. We first observe that $V$ is directed.
Indeed, if $u_1 \colon \mathbf{R}_1 \rightarrow \mathbf{Q}$ and $u_2 \colon \mathbf{R}_2 \rightarrow \mathbf{Q}$ in $V$,
 define $\mathbf{P}=(P,\leq)$
by adjoining a fresh bottom $\bot$ and a fresh top $\top$
to the disjoint union of $\mathbf{R}_1$ and $\mathbf{R}_2$.
It is easy to check that $\mathbf{P}$ is a lattice.
Define $u(y)=u_j(y)$ iff $y \in R_j$ for $j=1,2$,
$u(\bot)=x$ and $u(\top)=y$.  Since $x$ and $y$ are minimal and maximal in $\mathbf{Q}$ respectively, $u$ is a monotone map from $\mathbf{P}$ into $\mathbf{Q}$ and $u \in V$.
For $j=1,2$, let $f_j \colon \mathbf{R}_j \to \mathbf{P}$ in $\FP$
be the injection of $\mathbf{R}_j$ into $\mathbf{P}$.
Then $u_j=u \circ f_j$ for $j=1,2$, which proves that $V$ is directed.

Since $V$ is a directed preordered set with the inherited order of $ U_{\BDL}(\mathbf{Q})$, by Lemma \ref{Lemma:DirectedType}, ${\rm type}(V)\in\{0,1\}$.
We show that $\mathrm{type}(V) \neq 1$.
For every $n \in \mathbb{N}$, we define a unifier $u_n \colon \mathbf{T}_n \to \mathbf{Q}$ in $V$ as follows.
For $\mathbf{T}_n=(T_n,\leq) \in \FP$ we let
$$T_n=\{\bot,\top,1,\dots,n,j\cdot k \mid \text{$j<k$ in $\{1,\dots,n\}$ and $j+k$ is odd} \}\text{;}$$
here, $T_n$ is a formal language over the alphabet $\{\bot,\top,\cdot,1,\dots,n\}$.
The partial order over $T_n$ is defined by the following cover relation, where $j,k \in \{1,\dots,n\}$:
\begin{itemize}
\item[]  $\bot \prec j$;
\item[] $j,k \prec j\cdot k$ for all $j\cdot k \in T_n$; 
\item[] $j\cdot k \prec \top$ for all $j\cdot k \in T_n$;
\end{itemize} 
where $j,k \in \{1,\dots,n\}$.

Then $\mathbf{T}_n$ is a lattice.
See Figure~\ref{fig:posdl} for the Hasse diagram of $\mathbf{T}_5$.
\begin{figure}
\centering
\begin{picture}(0,0)%
\includegraphics{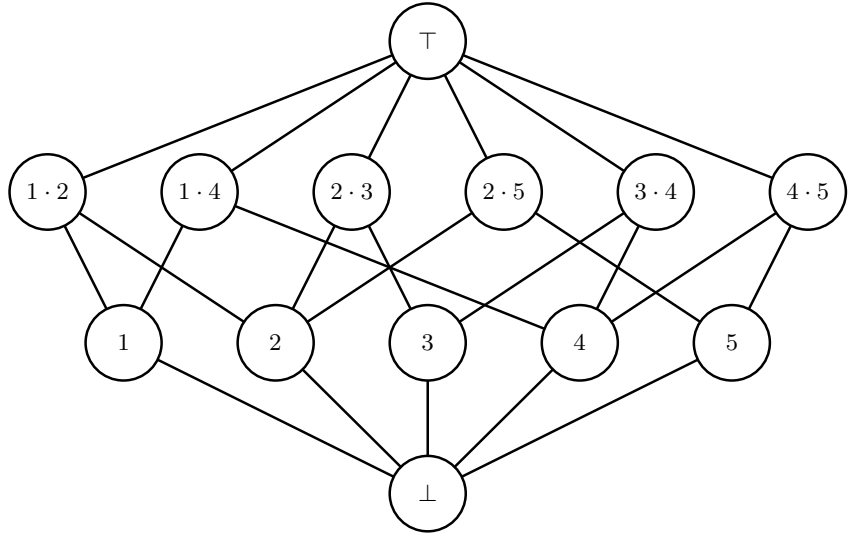}%
\end{picture}%
\setlength{\unitlength}{4144sp}%
\begingroup\makeatletter\ifx\SetFigFont\undefined%
\gdef\SetFigFont#1#2#3#4#5{%
  \reset@font\fontsize{#1}{#2pt}%
  \fontfamily{#3}\fontseries{#4}\fontshape{#5}%
  \selectfont}%
\fi\endgroup%
\begin{picture}(4980,3180)(4486,-5701)
\put(4726,-3706){\makebox(0,0)[b]{\smash{{\SetFigFont{9}{10.8}{\familydefault}{\mddefault}{\updefault}{\color[rgb]{0,0,0}$1 \cdot 2$}%
}}}}
\put(5626,-3706){\makebox(0,0)[b]{\smash{{\SetFigFont{9}{10.8}{\familydefault}{\mddefault}{\updefault}{\color[rgb]{0,0,0}$1 \cdot 4$}%
}}}}
\put(6526,-3706){\makebox(0,0)[b]{\smash{{\SetFigFont{9}{10.8}{\familydefault}{\mddefault}{\updefault}{\color[rgb]{0,0,0}$2 \cdot 3$}%
}}}}
\put(7426,-3706){\makebox(0,0)[b]{\smash{{\SetFigFont{9}{10.8}{\familydefault}{\mddefault}{\updefault}{\color[rgb]{0,0,0}$2 \cdot 5$}%
}}}}
\put(8326,-3706){\makebox(0,0)[b]{\smash{{\SetFigFont{9}{10.8}{\familydefault}{\mddefault}{\updefault}{\color[rgb]{0,0,0}$3 \cdot 4$}%
}}}}
\put(9226,-3706){\makebox(0,0)[b]{\smash{{\SetFigFont{9}{10.8}{\familydefault}{\mddefault}{\updefault}{\color[rgb]{0,0,0}$4 \cdot 5$}%
}}}}
\put(6976,-2806){\makebox(0,0)[b]{\smash{{\SetFigFont{9}{10.8}{\familydefault}{\mddefault}{\updefault}{\color[rgb]{0,0,0}$\top$}%
}}}}
\put(6976,-5506){\makebox(0,0)[b]{\smash{{\SetFigFont{9}{10.8}{\familydefault}{\mddefault}{\updefault}{\color[rgb]{0,0,0}$\bot$}%
}}}}
\put(5176,-4606){\makebox(0,0)[b]{\smash{{\SetFigFont{9}{10.8}{\familydefault}{\mddefault}{\updefault}{\color[rgb]{0,0,0}$1$}%
}}}}
\put(6076,-4606){\makebox(0,0)[b]{\smash{{\SetFigFont{9}{10.8}{\familydefault}{\mddefault}{\updefault}{\color[rgb]{0,0,0}$2$}%
}}}}
\put(6976,-4606){\makebox(0,0)[b]{\smash{{\SetFigFont{9}{10.8}{\familydefault}{\mddefault}{\updefault}{\color[rgb]{0,0,0}$3$}%
}}}}
\put(7876,-4606){\makebox(0,0)[b]{\smash{{\SetFigFont{9}{10.8}{\familydefault}{\mddefault}{\updefault}{\color[rgb]{0,0,0}$4$}%
}}}}
\put(8776,-4606){\makebox(0,0)[b]{\smash{{\SetFigFont{9}{10.8}{\familydefault}{\mddefault}{\updefault}{\color[rgb]{0,0,0}$5$}%
}}}}
\end{picture}%

\caption{$\mathbf{T}_5$ in Lemma~\ref{lem:NullBDL}.}\label{fig:posdl}
\end{figure}

We define $u_n \colon \mathbf{T}_n \to \mathbf{Q}$ as follows, where $j,k \in \{1,\dots,n\}$: 
\begin{enumerate}
\item[] $u_n(\bot)=x$ and $u_n(\top)=y$;
\item[] $u_n(j)=a$ and $u_n(j\cdot k)=c$, for all $j,j\cdot k \in T_n$ with $j$ odd;
\item[] $u_n(j)=b$ and $u_n(j\cdot k)=d$, for all $j, j\cdot k \in T_n$ with $j$ even. 
\end{enumerate}

Since for each $n\in\{1,2,\ldots\}$, $u_n \colon \mathbf{T}_n \to \mathbf{Q}$ is a monotone map, $\mathbf{T}_n$ is a lattice and $x,y\in u_n(\mathbf{T}_n)$, then $u_n$ is a unifier for $\mathbf{Q}$ in $V$.

Let $u \colon \mathbf{P} \to \mathbf{Q}$ be a unifier in $V$. We show that $u_n \leq u$ implies $|P| \geq n$.
Let $u_n=u \circ f$.  We claim that $f(j) \neq f(k)$ for all $j<k$ with $j,k \in \{1,\dots,n\}$.  
The claim is clear of $j$ and $k$ have different parity.  If $j$ and $k$ have the same parity, without loss of generality assume $j$ and $k$ are both odd, 
then let $l$ be even such that $j<l<k$.
By construction $j,l \leq j\cdot l$, then 
we have $f(j),f(l) \leq f(j\cdot l)$.  Since $\mathbf{P}$ is a lattice,
$$f(j),f(l) \leq f(j) \vee f(l) \leq f(j\cdot l)\text{.}$$
Similarly,
$$f(l),f(k) \leq f(l) \vee f(k) \leq f(l\cdot k)\text{.}$$
Assume for a contradiction that $f(j)=f(k)$.  Then,
$$f(j)=f(k),f(l) \leq f(j) \vee f(l)=f(l) \vee f(k) \leq f(j\cdot l),f(l\cdot k)\text{,}$$
and applying $u$ through, since $u_n=u \circ f$,
$$a,b \leq u(f(l) \vee f(k)) \leq c,d\text{,}$$
contradicting $(ii)$.  Therefore, a most general unifier $u \colon \mathbf{P} \to \mathbf{Q}$ 
has $|P| \geq n$ for every $n \in \mathbb{N}$, impossible because $\mathbf{P}$ is finite.  Thus, $\mathrm{type}(V)\neq 1$.  

Then ${\rm type}(V)=0$ and by Lemma~\ref{Lemma:UpsetNullary}, $\mathrm{type}_{\BDL}(\mathbf{Q})=0$, as desired.
\end{proof}

\begin{theorem}\label{Thm:UnifLat}
Let $\mathbf{P}=(P,\leq) \in \FP$ be a solvable instance of $\textsc{Unif}(\BDL)$.  Then:
\begin{align*}
\mathrm{type}_{\BDL}(\mathbf{P}) &=
\begin{cases}
1\text{,} & \text{iff $\mathbf{P}$ is a lattice;}\\
\omega\text{,} & \text{iff $\mathbf{P}$ is not a lattice,}\\
               & \text{but $[x,y]$ is a lattice for all $x \leq y$ in $\mathbf{P}$;}\\
0\text{,} & \text{otherwise.}
\end{cases}
\end{align*}
\end{theorem}
\begin{proof}
If $\mathbf{P}$ is a lattice, then $\mathrm{type}_{\BDL}(\mathbf{P})=1$ because 
$\mathrm{id}_{\mathbf{P}}$ is a most general unifier for $\mathbf{P}$.

Suppose that $\mathbf{P}$ is not a lattice and $[x,y]$ is a lattice for all $x \leq y$ in $\mathbf{P}$.  Define,
for every $x,y \in P$ such that $x \leq y$, $x\in{\rm min}(\mathbf{P})$, and $y\in{\rm max}(\mathbf{P})$, 
the (inclusion) unifier $u_{x,y} \colon [x,y] \to \mathbf{P}$ by $u_{x,y}(z)=z$ for all $z \in [x,y]$.  
Clearly, there are finitely many unifiers of the form $u_{x,y}$ with $x \leq y$ in $\mathbf{P}$,
because $P$ is finite. 
%(and at least one, because $P$ is nonempty).  
We claim that they form
% unifiers $u_{x,y}$ with $x \leq y$ in $\mathbf{P}$ form 
a $\mu$-set in $U_{\BDL}(\mathbf{P})$. Since $x\in{\rm min}(\mathbf{P})$, and $y\in{\rm max}(\mathbf{P})$, 
any two unifiers of the type $u_{x,y}$ and $u_{x',y'}$ are comparable iff $x=x'$ and $y=y'$.  
Now let $u \colon \mathbf{L} \to \mathbf{P}$ be a unifier for $\mathbf{P}$.
Now, $\mathbf{L}$ is bounded, with bottom $\bot$ and top $\top$. Let $x\in{\rm min}(\mathbf{P})$ and $y\in{\rm max}(\mathbf{P})$ be such that 
$x\leq u(\bot)\leq u(\top)\leq y$. Then $u(L) \subseteq [x,y]$ and $u_{x,y} \circ u=u$, so that $u_{x,y}$ is more general than
$u$.  Thus $\mathrm{type}_{\BDL}(\mathbf{P}) \in \{1,\omega\}$.

Since $\mathbf{P}$ is not a lattice but for each $x\leq y$ in $\mathbf{P}$, $[x,y]$ is a lattice, then $\mathbf{P}$ cannot be bounded. 
Assume that $x_1 \neq x_2$ are minimal points in $\mathbf{P}$
(the argument is similar for maximal points).  Then let $\mathbf{L}=(\{p\},\leq)$,
for $i=1,2$ let $u_i \colon \mathbf{L} \to \mathbf{P}$ be the unifier such that $u_i(p)=x_i$.  
Suppose for a contradiction that there exists a unifier $u \colon \mathbf{M} \to \mathbf{P}$ such that $u_1,u_2\leq u$.  
For $i=1,2$, let $u \circ f_i=u_i$ be a factorization of $u_i$
where $u \colon \mathbf{M} \to \mathbf{P}$.  Then by monotonicity 
$u(f_1(p) \wedge f_2(p)) \leq u(f_1(p)),u(f_2(p))$, 
and since $x_1 \parallel x_2$, we have $u(f_1(p) \wedge f_2(p))<x_1, x_2$ 
which contradicts the minimality of $x_1$ and $x_2$.  
Thus $u_1$ and $u_2$ have no common upper bound in $U_{\BDL}(\mathbf{P})$, and $\mathrm{type}_{\BDL}(\mathbf{P}) \neq 1$.
This concludes the proof that $\mathrm{type}_{\BDL}(\mathbf{P})=\omega$.

Finally, let $x \leq y$ in $P$ be such that $[x,y]$ is not a lattice.  Then
there exist $a,b,c,d \in P$ such that $x \leq a,b \leq c,d \leq y$
and there does not exist $e \in P$ such that $a,b \leq e \leq c,d$.
By Lemma \ref{lem:NullBDL}, it follows that $\mathrm{type}_{\BDL}(\mathbf{P})=0$.
\end{proof}

\subsection{Kleene Algebras}\label{subsect:unifk}

We provide a complete classification of solvable instances of the unification problem 
over Kleene algebras (Theorem \ref{Theo:UnifClassKleene}). 
%establishing in particular that Kleene algebras have nullary unification type.  
Combining the algebraic unification theory by Ghilardi \cite{G97}, 
Theorem~\ref{th:kmain}, and Theorem~\ref{th:dualities}, 
the unification problem over Kleene algebras reduces to the following combinatorial question:

\begin{description}
\item[Problem]  $\textsc{Unif}(\K)$.
\item[Instance]  $\mathbf{Q}=(Q,\leq,i) \in \FPK$.
\item[Solution]  A morphism $u \colon \mathbf{P} \to \mathbf{Q}$ in $\FPK$,
where $\mathbf{P}$ satisfies $(K_1)$, $(K_2)$, $(M_2)$, and $(M_3)$.
\end{description}

\begin{remark}\label{Rem:KleeneSolvable}
An instance $\mathbf{Q}=(Q,\leq,i)$ of $\textsc{Unif}(\K)$ is solvable iff
 $\{ x \in Q \mid x=i(x) \}\neq \emptyset$.  Indeed, if $\mathbf{P} \in \FPK$ satisfies $(K_1)$ and $(M_2)$, 
and $\mathbf{Q}$ admits a morphism from $\mathbf{P}$, then $\mathbf{Q}$ is $\{ x \in Q \mid x=i(x) \}\neq \emptyset$.
Conversely, if $\mathbf{Q}$ is such that $\{ x \in Q \mid x=i(x) \}\neq \emptyset$, then $\mathbf{Q}$ admits a morphism from
$\mathbf{P}=(P,\leq,i)$ where $P=\{x\}$, and $i(x)=x$;
clearly, $\mathbf{P}$ satisfies $(K_1)$, $(K_2)$, $(M_2)$, and $(M_3)$.
\end{remark}

Given a solvable instance $\mathbf{Q}$ of $\textsc{Unif}(\K)$, we let $U_{\K}(\mathbf{Q})$ 
denote the preordered set of unifiers of $\mathbf{Q}$, which is defined as in Section~\ref{subsect:unifbdl}.  

We now embark in the proof of the main result of this section.  The structure of the 
proof is the following: we identify two sufficient conditions for an instance of the unification problem 
to have nullary type (Lemma~\ref{lemma:unifkcase1} and Lemma~\ref{lemma:unifkcase2}), 
and then we prove that the identified conditions are indeed necessary for nullarity (Theorem~\ref{Theo:UnifClassKleene}).

We first establish the following fact for later use.

\begin{lemma}\label{Lem:UpKleene}
Let $\mathbf{Q}=(Q,\leq,i) \in \FPK$ be an instance of $\textsc{Unif}(\K)$ and $x\in Q$ be a minimal element of $\mathbf{Q}$.
Then 
\begin{equation}\label{Eq:UpwardKleene}
  V=\{ u \colon \mathbf{P} \rightarrow \mathbf{Q} \in U_{\K}(\mathbf{Q}) \mid x \in u(P)\} 
\end{equation}
is a directed upset in $U_{\K}(\mathbf{Q})$.
\end{lemma}
\begin{proof}
Clearly, $V$ is an upset in $U_{\K}(\mathbf{Q})$.  
If $V$ is empty, directedness is trivial.  Otherwise, 
there exists $y\in Q$ such that $x\leq y=i(y)$, which proves that $x\leq i(x)$. 
Let $u_1 \colon \mathbf{R}_1 \rightarrow \mathbf{Q}$ and $u_2 \colon \mathbf{R}_2 \rightarrow \mathbf{Q}$ in $V$, 
with $\mathbf{R}_j=(R_j,\leq_j,i_j)$ for $j=1,2$.  
Define $\mathbf{P}=(P,\leq,i)$ by adjoining a fresh bottom $\bot$ and a fresh top $\top$
to the disjoint union of $\mathbf{R}_1$ and $\mathbf{R}_2$,
and by letting $i(\bot)=\top$, $i(\top)=\bot$,
and $i(y)=i_j(y)$ iff $y \in R_j$ for $j=1,2$.
Since $\mathbf{R}_1$ and $\mathbf{R}_2$ satisfy $(K_1)$, $(K_2)$, $(M_2)$, and $(M_3)$, so does $\mathbf{P}$.
Let $u\colon P\to Q$ be the map defined by: $u(y)=u_j(y)$ iff $y \in R_j$ for $j=1,2$,
$u(\bot)=x$ and $u(\top)=i(x)$. It follows that $u \in V$.
For $j=1,2$, let $f_j \colon \mathbf{R}_j \to \mathbf{P}$ in $\FPK$
be the injection of $\mathbf{R}_j$ into $\mathbf{P}$.
Then $u=u_j \circ f_j$ for $j=1,2$, as desired.
\end{proof}

\begin{lemma}\label{lemma:unifkcase1}
Let $\mathbf{Q}=(Q,\leq,i) \in \FPK$ be an instance of $\textsc{Unif}(\K)$.
If there exist $x,a,b,c,d,y,z\in Q$ such that:
\begin{itemize}
\item[$(i)$] $x\leq a,b \leq c,d$;
\item[$(ii)$] $c\leq y=i(y)$; $d\leq z=i(z)$;
\item[$(iii)$] there does not exist $e\in Q$ such that $a,b \leq e \leq c,d$;
\end{itemize}
then $\mathrm{type}_{\K}(\mathbf{Q})=0$ (see Figure~\ref{fig:kconf1}).
\begin{figure}
\centering
\begin{picture}(0,0)%
\includegraphics{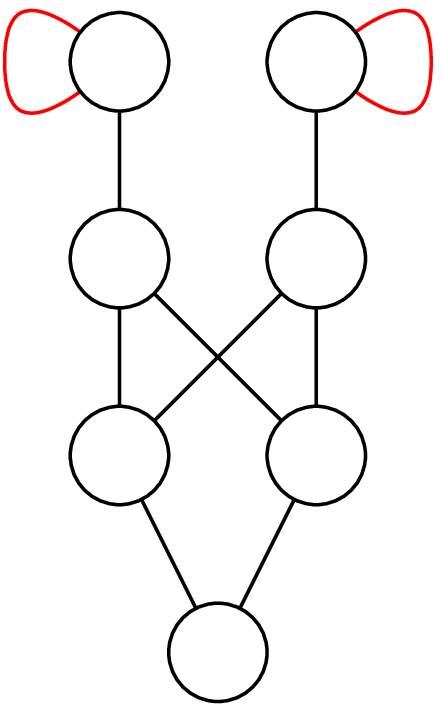}%
\end{picture}%
\setlength{\unitlength}{4144sp}%
\begingroup\makeatletter\ifx\SetFigFont\undefined%
\gdef\SetFigFont#1#2#3#4#5{%
  \reset@font\fontsize{#1}{#2pt}%
  \fontfamily{#3}\fontseries{#4}\fontshape{#5}%
  \selectfont}%
\fi\endgroup%
\begin{picture}(1994,3197)(5979,-5701)
\put(6977,-5508){\makebox(0,0)[b]{\smash{{\SetFigFont{9}{10.8}{\familydefault}{\mddefault}{\updefault}{\color[rgb]{0,0,0}$x$}%
}}}}
\put(7430,-3714){\makebox(0,0)[b]{\smash{{\SetFigFont{9}{10.8}{\familydefault}{\mddefault}{\updefault}{\color[rgb]{0,0,0}$d$}%
}}}}
\put(6530,-3713){\makebox(0,0)[b]{\smash{{\SetFigFont{9}{10.8}{\familydefault}{\mddefault}{\updefault}{\color[rgb]{0,0,0}$c$}%
}}}}
\put(7434,-4613){\makebox(0,0)[b]{\smash{{\SetFigFont{9}{10.8}{\familydefault}{\mddefault}{\updefault}{\color[rgb]{0,0,0}$b$}%
}}}}
\put(6529,-4604){\makebox(0,0)[b]{\smash{{\SetFigFont{9}{10.8}{\familydefault}{\mddefault}{\updefault}{\color[rgb]{0,0,0}$a$}%
}}}}
\put(7424,-2806){\makebox(0,0)[b]{\smash{{\SetFigFont{9}{10.8}{\familydefault}{\mddefault}{\updefault}{\color[rgb]{0,0,0}$z$}%
}}}}
\put(6524,-2806){\makebox(0,0)[b]{\smash{{\SetFigFont{9}{10.8}{\familydefault}{\mddefault}{\updefault}{\color[rgb]{0,0,0}$y$}%
}}}}
\end{picture}%

\caption{Subposet of $\mathbf{Q}$ in Lemma~\ref{lemma:unifkcase1}.}\label{fig:kconf1}
\end{figure}
\end{lemma}
\begin{proof}
Since $\mathbf{Q}$ is a finite poset, we assume without loss of generality $x \in \mathrm{min}(\mathbf{Q})$.
By $(iii)$, we have $a\neq b$ and $c\neq d$.  Let,
$$V=\{ u \colon \mathbf{P} \rightarrow \mathbf{Q} \in U_{\K}(\mathbf{Q}) \mid x \in u(P)\}\text{.}$$

By Lemma \ref{Lem:UpKleene} $V$ is an directed upset of $U_{\K}(\mathbf{Q})$.
By Lemma \ref{Lemma:UpsetNullary},  to prove that ${\rm type}(U_{\K}(\mathbf{Q}))=0$ it is enough to prove that ${\rm type}(V)=0$.
Since $V$ is directed, by Lemma \ref{Lemma:DirectedType}, $\mathrm{type}(V)\in\{0,1\}$.
We show that $\mathrm{type}(V) \neq 1$.
For every $n \in \mathbb{N}$, we define a unifier $u_n \colon \mathbf{T}_n \to \mathbf{Q}$ in $V$ as follows.
For $\mathbf{T}_n=(T_n,\leq,i) \in \FPM$ we let
$$T_n=\{\bot,\overline{\bot},j,\overline{j},j\cdot k,\overline{j\cdot k},j\diamond k \mid \text{$j<k$ in $\{1,\dots,n\}$ and $j + k$ is odd} \}\text{;}$$
here, $T_n$ is a formal language over the alphabet $A \cup \{ \overline{s} \mid s \in A\}$,
with $A=\{\bot,\cdot ,\diamond,1,\dots,n\}$.
The map $i \colon T_n \to T_n$ is defined as follows, where $j,k \in \{1,\dots,n\}$ 
and $y \in \{\bot,j,j\cdot k \mid \text{$j<k$ in $\{1,\dots,n\}$ and $j + k$ is odd} \} \subseteq T_n$:

\begin{itemize}
\item[] $i(j\diamond k)=j\diamond k$ for all $j\diamond k \in T_n$;
\item[]$i(y)=\overline{y}$ and $i(\overline{y})=y$ for all $y,\overline{y} \in T_n$.
\end{itemize}

The partial order over $T_n$ is defined by the following cover relation, for all $j,k \in \{1,\dots,n\}$:
\begin{itemize}
\item[] $\bot \prec j$ and $i(j) \prec i(\bot)$, for all $j \in T_n$
\item[] $j,k \prec j\cdot k$ and $i(j\cdot k) \prec i(j),i(k)$, for all $j,k,j\cdot k \in T_n$.
\end{itemize}

It is easy to check that $\mathbf{T}_n$ satisfies $(K_1)$, $(K_2)$, $(M_2)$, and $(M_3)$.
Figure~\ref{fig:posk1} provides the Hasse diagram of $\mathbf{T}_4$.
\begin{figure}
\centering
\begin{picture}(0,0)%
\includegraphics{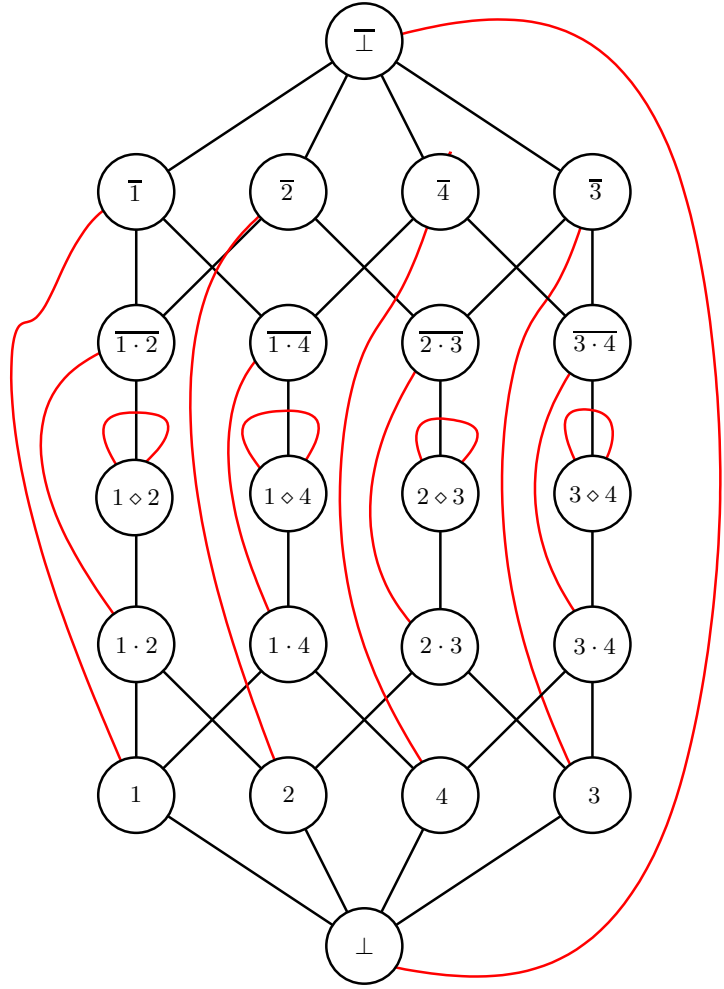}%
\end{picture}%
\setlength{\unitlength}{4144sp}%
\begingroup\makeatletter\ifx\SetFigFont\undefined%
\gdef\SetFigFont#1#2#3#4#5{%
  \reset@font\fontsize{#1}{#2pt}%
  \fontfamily{#3}\fontseries{#4}\fontshape{#5}%
  \selectfont}%
\fi\endgroup%
\begin{picture}(4241,5880)(4864,-5701)
\put(6977,-5508){\makebox(0,0)[b]{\smash{{\SetFigFont{9}{10.8}{\familydefault}{\mddefault}{\updefault}{\color[rgb]{0,0,0}$\bot$}%
}}}}
\put(5625,-4602){\makebox(0,0)[b]{\smash{{\SetFigFont{9}{10.8}{\familydefault}{\mddefault}{\updefault}{\color[rgb]{0,0,0}$1$}%
}}}}
\put(6529,-4604){\makebox(0,0)[b]{\smash{{\SetFigFont{9}{10.8}{\familydefault}{\mddefault}{\updefault}{\color[rgb]{0,0,0}$2$}%
}}}}
\put(8340,-3717){\makebox(0,0)[b]{\smash{{\SetFigFont{9}{10.8}{\familydefault}{\mddefault}{\updefault}{\color[rgb]{0,0,0}$3 \cdot 4$}%
}}}}
\put(7430,-3714){\makebox(0,0)[b]{\smash{{\SetFigFont{9}{10.8}{\familydefault}{\mddefault}{\updefault}{\color[rgb]{0,0,0}$2 \cdot 3$}%
}}}}
\put(6530,-3713){\makebox(0,0)[b]{\smash{{\SetFigFont{9}{10.8}{\familydefault}{\mddefault}{\updefault}{\color[rgb]{0,0,0}$1 \cdot 4$}%
}}}}
\put(5628,-1913){\makebox(0,0)[b]{\smash{{\SetFigFont{9}{10.8}{\familydefault}{\mddefault}{\updefault}{\color[rgb]{0,0,0}$\overline{1 \cdot 2}$}%
}}}}
\put(8340,-1910){\makebox(0,0)[b]{\smash{{\SetFigFont{9}{10.8}{\familydefault}{\mddefault}{\updefault}{\color[rgb]{0,0,0}$\overline{3 \cdot 4}$}%
}}}}
\put(6977,-119){\makebox(0,0)[b]{\smash{{\SetFigFont{9}{10.8}{\familydefault}{\mddefault}{\updefault}{\color[rgb]{0,0,0}$\overline{\bot}$}%
}}}}
\put(6530,-1914){\makebox(0,0)[b]{\smash{{\SetFigFont{9}{10.8}{\familydefault}{\mddefault}{\updefault}{\color[rgb]{0,0,0}$\overline{1 \cdot 4}$}%
}}}}
\put(7430,-1913){\makebox(0,0)[b]{\smash{{\SetFigFont{9}{10.8}{\familydefault}{\mddefault}{\updefault}{\color[rgb]{0,0,0}$\overline{2 \cdot 3}$}%
}}}}
\put(5628,-3714){\makebox(0,0)[b]{\smash{{\SetFigFont{9}{10.8}{\familydefault}{\mddefault}{\updefault}{\color[rgb]{0,0,0}$1 \cdot 2$}%
}}}}
\put(6378,-2825){\makebox(0,0)[lb]{\smash{{\SetFigFont{9}{10.8}{\familydefault}{\mddefault}{\updefault}{\color[rgb]{0,0,0}$1\diamond 4$}%
}}}}
\put(7288,-2821){\makebox(0,0)[lb]{\smash{{\SetFigFont{9}{10.8}{\familydefault}{\mddefault}{\updefault}{\color[rgb]{0,0,0}$2\diamond 3$}%
}}}}
\put(5481,-2829){\makebox(0,0)[lb]{\smash{{\SetFigFont{9}{10.8}{\familydefault}{\mddefault}{\updefault}{\color[rgb]{0,0,0}$1\diamond 2$}%
}}}}
\put(8188,-2811){\makebox(0,0)[lb]{\smash{{\SetFigFont{9}{10.8}{\familydefault}{\mddefault}{\updefault}{\color[rgb]{0,0,0}$3\diamond 4$}%
}}}}
\put(8338,-4610){\makebox(0,0)[b]{\smash{{\SetFigFont{9}{10.8}{\familydefault}{\mddefault}{\updefault}{\color[rgb]{0,0,0}$3$}%
}}}}
\put(7437,-4612){\makebox(0,0)[b]{\smash{{\SetFigFont{9}{10.8}{\familydefault}{\mddefault}{\updefault}{\color[rgb]{0,0,0}$4$}%
}}}}
\put(8344,-1006){\makebox(0,0)[b]{\smash{{\SetFigFont{9}{10.8}{\familydefault}{\mddefault}{\updefault}{\color[rgb]{0,0,0}$\overline{3}$}%
}}}}
\put(7443,-1010){\makebox(0,0)[b]{\smash{{\SetFigFont{9}{10.8}{\familydefault}{\mddefault}{\updefault}{\color[rgb]{0,0,0}$\overline{4}$}%
}}}}
\put(6517,-1011){\makebox(0,0)[b]{\smash{{\SetFigFont{9}{10.8}{\familydefault}{\mddefault}{\updefault}{\color[rgb]{0,0,0}$\overline{2}$}%
}}}}
\put(5618,-1013){\makebox(0,0)[b]{\smash{{\SetFigFont{9}{10.8}{\familydefault}{\mddefault}{\updefault}{\color[rgb]{0,0,0}$\overline{1}$}%
}}}}
\end{picture}%

\caption{$\mathbf{T}_4$ in Lemma~\ref{lemma:unifkcase1}.}\label{fig:posk1}
\end{figure}

For $j,k \in \{1,\dots,n\}$, we define $u_n \colon \mathbf{T}_n \to \mathbf{Q}$ by putting, 
% \begin{enumerate}
% \item[] $u_n(\bot)=x$ and $u_n(i(\bot))=i(x)$;
% \item[] $u_n(j)=a$ and $u_n(i(j))=i(a)$ for all odd $j \in T_n$;
% \item[] $u_n(j)=b$ and $u_n(i(j))=i(b)$ for all even $j \in T_n$;
% \item[] $u_n(j\cdot k)=c$ and $u_n(i(j\cdot k))=i(c)$ for all $j\cdot k \in T_n$ with $j$ odd;
% \item[] $u_n(j\cdot k)=d$ and $u_n(i(j\cdot k))=i(d)$ for all $j\cdot k \in T_n$ with $j$ even;
% \item[] $u_n(j\diamond k)=u_n(i(j\diamond k))=y$ for all $j\diamond k \in T_n$ with $j$ odd;
% \item[] $u_n(j\diamond k)=u_n(i(j\diamond k))=z$ for all $j\diamond k \in T_n$ with $j$ even.
% \end{enumerate}
\begin{enumerate}
\item[] $u_n(\bot)=x$;
\item[] $u_n(j)=a$, $u_n(j\cdot k)=c$, $u_n(j\diamond k)=y$, for all $j,j\cdot k,j\diamond k \in T_n$ with $j$ odd; 
\item[] $u_n(j)=b$, $u_n(j\cdot k)=d$, $u_n(j\diamond k)=z$, for all $j,j\cdot k,j\diamond k \in T_n$ with $j$ even;
\end{enumerate}
and, for all $y \in \{\bot,j,j\cdot k,j\diamond k \mid \text{$j<k$ in $\{1,\dots,n\}$ and $j + k$ is odd} \} \subseteq T_n$, 
\begin{enumerate}
\item[] $u_n(i(y))=i(u_n(y))$.
\end{enumerate}
It follows by a straightforward computation that $u_n \colon \mathbf{T}_n \to \mathbf{Q}$ is a morphism in $\FPK$. Therefore $u_n$ is a unifier for $\mathbf{Q}$ in $V$ for each $n\in\mathbb{N}$.

Let $u \colon \mathbf{P} \to \mathbf{Q}$ be a unifier for $\mathbf{Q}$. We show that $u_n \leq u$ implies $|P| \geq n$.
Let $u_n=u \circ f$.  We claim that $f(j) \neq f(k)$ for all $j<k$ with $j,k \in T_n$.
If $j + k$ is odd, it is straightforward.  If $j+k$ is even, without loss of generality assume $j,k$ both odd.
Then let $l$  be an even number such that $j<l<k$.
By construction $j,l \leq j\cdot l \leq i(j\cdot l)$, then
we have $f(j),f(l) \leq f(j\cdot l) \leq i(f(j\cdot l))$.  By $(K_1)$, 
$f(j) \vee f(l)$ exists in $\mathbf{P}$ and it satisfies: $$f(j) \vee f(l) \leq i(f(j) \vee f(l))\text{.}$$
Then, $$f(j),f(l) \leq f(j) \vee f(l) \leq f(j\cdot l)\text{.}$$
Similarly, $f(l) \vee f(k)$ exists in $\mathbf{P}$ and it satisfies: 
$$f(l) \vee f(k) \leq i(f(l) \vee f(k))\text{.}$$
Then, 
$$f(l),f(k) \leq f(l) \vee f(k) \leq f(l\cdot k)\text{.}$$
By way of contradiction assume $f(j)=f(k)$, then
$$f(j)=f(k),f(l) \leq f(j) \vee f(l)=f(l) \vee f(k) \leq f(j\cdot l),f(l\cdot k)\text{,}$$
and applying $u$ through, since $u_n=u \circ f$,
$$a,b \leq u(f(l) \vee f(k)) \leq c,d\text{,}$$
which contradicts  $(iii)$.  Therefore, $|P| \geq n$. 

This proves that $\mathrm{type}(V)\neq 1$. Then $\mathrm{type}(V)=0$. By Lemma~\ref{Lemma:UpsetNullary}, 
$\mathrm{type}_{\K}(\mathbf{Q})=0$, as desired.
\end{proof}

\begin{lemma}\label{lemma:unifkcase2}
Let $\mathbf{Q}=(Q,\leq,i)$ be an instance of $\textsc{Unif}(\K)$.
If there exist $x,a,b,c,d,e,f,y,z,w\in Q$ such that:
\begin{itemize}
\item[$(i)$] $x\leq a,b,c$; $a\leq d,e$; $b\leq d,f$; $c\leq e,f$;  
\item[$(ii)$] $d\leq y=i(y)$; $e\leq z=i(z)$; $f\leq w=i(w)$; 
\item[$(iii)$] there does not exist $g \in Q$ such that $a,b,c \leq g\leq i(g)$,
\end{itemize}
then $\mathrm{type}_{\K}(\mathbf{Q})=0$ (see Figure~\ref{fig:kconf2}).
\begin{figure}
\centering
\begin{picture}(0,0)%
\includegraphics{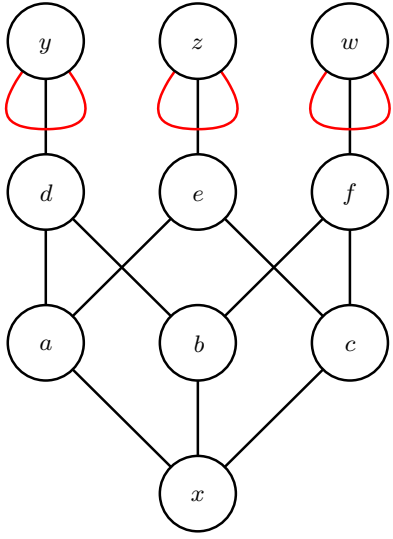}%
\end{picture}%
\setlength{\unitlength}{4144sp}%
\begingroup\makeatletter\ifx\SetFigFont\undefined%
\gdef\SetFigFont#1#2#3#4#5{%
  \reset@font\fontsize{#1}{#2pt}%
  \fontfamily{#3}\fontseries{#4}\fontshape{#5}%
  \selectfont}%
\fi\endgroup%
\begin{picture}(2314,3180)(6269,-5701)
\put(7430,-3714){\makebox(0,0)[b]{\smash{{\SetFigFont{9}{10.8}{\familydefault}{\mddefault}{\updefault}{\color[rgb]{0,0,0}$e$}%
}}}}
\put(6530,-3713){\makebox(0,0)[b]{\smash{{\SetFigFont{9}{10.8}{\familydefault}{\mddefault}{\updefault}{\color[rgb]{0,0,0}$d$}%
}}}}
\put(7434,-4613){\makebox(0,0)[b]{\smash{{\SetFigFont{9}{10.8}{\familydefault}{\mddefault}{\updefault}{\color[rgb]{0,0,0}$b$}%
}}}}
\put(6529,-4604){\makebox(0,0)[b]{\smash{{\SetFigFont{9}{10.8}{\familydefault}{\mddefault}{\updefault}{\color[rgb]{0,0,0}$a$}%
}}}}
\put(7424,-2806){\makebox(0,0)[b]{\smash{{\SetFigFont{9}{10.8}{\familydefault}{\mddefault}{\updefault}{\color[rgb]{0,0,0}$z$}%
}}}}
\put(6524,-2806){\makebox(0,0)[b]{\smash{{\SetFigFont{9}{10.8}{\familydefault}{\mddefault}{\updefault}{\color[rgb]{0,0,0}$y$}%
}}}}
\put(8327,-2808){\makebox(0,0)[b]{\smash{{\SetFigFont{9}{10.8}{\familydefault}{\mddefault}{\updefault}{\color[rgb]{0,0,0}$w$}%
}}}}
\put(8327,-3708){\makebox(0,0)[b]{\smash{{\SetFigFont{9}{10.8}{\familydefault}{\mddefault}{\updefault}{\color[rgb]{0,0,0}$f$}%
}}}}
\put(8327,-4608){\makebox(0,0)[b]{\smash{{\SetFigFont{9}{10.8}{\familydefault}{\mddefault}{\updefault}{\color[rgb]{0,0,0}$c$}%
}}}}
\put(7427,-5508){\makebox(0,0)[b]{\smash{{\SetFigFont{9}{10.8}{\familydefault}{\mddefault}{\updefault}{\color[rgb]{0,0,0}$x$}%
}}}}
\end{picture}%

\caption{Subposet of $\mathbf{Q}$ in Lemma~\ref{lemma:unifkcase2} and Lemma~\ref{lemma:unifMcase3}.}\label{fig:kconf2}
\end{figure}
\end{lemma}
\begin{proof}
Since $\mathbf{Q}$ is a finite poset, we assume without loss of generality that $x \in \mathrm{min}(\mathbf{Q})$.
By $(i)$ and $(iii)$, we have $|\{a,b,c\}|=|\{d,e,f\}|=3$.  Let,
$$V=\{ u \colon \mathbf{P} \rightarrow \mathbf{Q} \in U_{\K}(\mathbf{Q}) \mid x \in u(P)\}\text{.}$$

By Lemma \ref{Lem:UpKleene}, $V$ is an directed upset in $U_{\K}(\mathbf{Q})$.
Then Lemma \ref{Lemma:DirectedType}, proves ${\rm type}(V)\in\{0,1\}$.
We show that $\mathrm{type}(V) \neq 1$.
For every $n \in \mathbb{N}$, we define a unifier $u_n \colon \mathbf{T}_n \to \mathbf{Q}$ in $V$ as follows.  Let
$T_n=L \cup I \cup \overline{L}$ where,
\begin{align*}
L =& \{ \bot,j,j\cdot k,j\circ j\cdot k \mid \text{$j \neq k$ in $\{1,\dots,n\}$} \} \ \cup\\
& \{ j\cdot k\circ k\cdot j \mid \text{$j<k$ in $\{1,\dots,n\}$} \}\text{,} \\
\overline{L} =& \{ \overline{v} \mid v \in L \}\text{,}\\
I =& \{j \diamond j\cdot k  \mid \text{$j\neq k$ in $\{1,\dots,n\}$} \}\cup \{ j\cdot k \diamond k\cdot j \mid \text{$j<k$ in $\{1,\dots,n\}$} \}\text{;}
\end{align*}
here, $T_n$ is a formal language over $A \cup \{ \overline{s} \mid s \in A\}$ with $A=\{\bot,\circ,\diamond,\cdot ,1,\dots,n\}$.
The map $i \colon T_n \to T_n$ is defined by: 
\begin{enumerate}
\item[] $i(v)=v$ for all $v \in I$;
\item[] $i(v)=\overline{v}$ and $i(\overline{v})=v$ for all $v \in L$.
\end{enumerate}
The partial order over $T_n$ is defined by the cover relation 
containing the covers listed below, where $j, k \in \{1,\dots,n\}$: 
% \begin{enumerate}
% \item[] $\bot \prec j$ and $i(j) \prec i(\bot)$;
% \item[] $\bot \prec j\cdot k$ and $i(j\cdot k) \prec i(\bot)$;
% \item[] $j, j\cdot k\prec j\circ j\cdot k$ and $i(j\circ j\cdot k) \prec i(j),i(j\cdot k)$;
% \item[] $j\cdot k,k\cdot j \prec j\cdot k\circ k\cdot j$ and $i(j\cdot k\circ k\cdot j) \prec i(j\cdot k),i(k\cdot j)$ if $j<k$;
% \item[] $j\circ j\cdot k \prec j \diamond j\cdot k=i(j \diamond j\cdot k) \prec i(j\circ j\cdot k)$;
% \item[] $j\cdot k\circ k\cdot j \prec j\cdot k \diamond k\cdot j=i(j\cdot k \diamond k\cdot j) \prec i(j\cdot k\circ k\cdot j)$ if $j<k$.
% \end{enumerate}
\begin{enumerate}
\item[] $\bot \prec j,j\cdot k$ for all $j,j\cdot k \in T_n$;
\item[] $j, j\cdot k\prec j\circ j\cdot k$ for all $j,j\cdot k,j\circ j\cdot k \in T_n$;
\item[] $j\cdot k,k\cdot j \prec j\cdot k\circ k\cdot j$ for all $j\cdot k,k\cdot j \in T_n$;
\item[] $j\circ j\cdot k \prec j \diamond j\cdot k$ for all $j\circ j\cdot k, j \diamond j\cdot k \in T_n$;
\item[] $j\cdot k\circ k\cdot j \prec j\cdot k \diamond k\cdot j$ for all $j\cdot k\circ k\cdot j,j\cdot k \diamond k\cdot j \in T_n$;
\end{enumerate}
and, for each $x \prec y$ in the list, the cover 
\begin{enumerate}
\item[] $i(y) \prec i(x)$.
\end{enumerate}
It is easy to check that $\mathbf{T}_n$ satisfies $(M_1)$, $(M_2)$, and $(M_3)$.
Notice that $(M_1)$ implies $(K_1)$ and $(K_2)$. 
Figure~\ref{fig:posk2} provides the Hasse diagram of $\mathbf{T}_2$.
\begin{figure}
\centering
\begin{picture}(0,0)%
\includegraphics{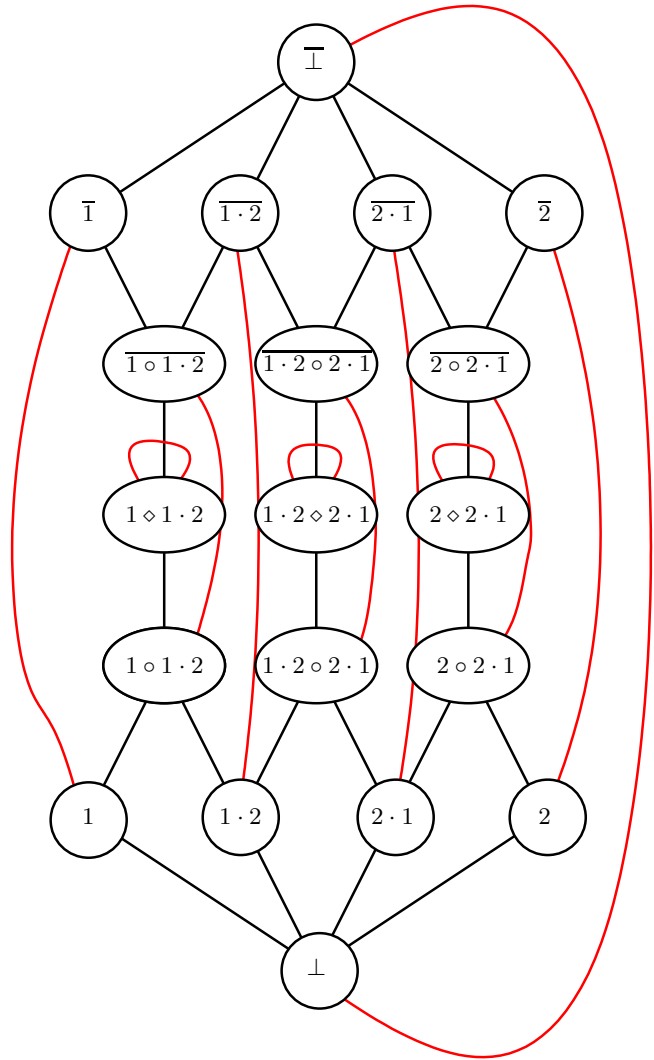}%
\end{picture}%
\setlength{\unitlength}{4144sp}%
\begingroup\makeatletter\ifx\SetFigFont\undefined%
\gdef\SetFigFont#1#2#3#4#5{%
  \reset@font\fontsize{#1}{#2pt}%
  \fontfamily{#3}\fontseries{#4}\fontshape{#5}%
  \selectfont}%
\fi\endgroup%
\begin{picture}(3825,6317)(10329,-7370)
\put(12154,-3251){\makebox(0,0)[b]{\smash{{\SetFigFont{9}{10.8}{\familydefault}{\mddefault}{\updefault}{\color[rgb]{0,0,0}$\overline{1 \cdot 2\circ 2 \cdot 1}$}%
}}}}
\put(10801,-5956){\makebox(0,0)[b]{\smash{{\SetFigFont{9}{10.8}{\familydefault}{\mddefault}{\updefault}{\color[rgb]{0,0,0}$1$}%
}}}}
\put(11701,-5956){\makebox(0,0)[b]{\smash{{\SetFigFont{9}{10.8}{\familydefault}{\mddefault}{\updefault}{\color[rgb]{0,0,0}$1 \cdot 2$}%
}}}}
\put(12151,-6856){\makebox(0,0)[b]{\smash{{\SetFigFont{9}{10.8}{\familydefault}{\mddefault}{\updefault}{\color[rgb]{0,0,0}$\bot$}%
}}}}
\put(13501,-2356){\makebox(0,0)[b]{\smash{{\SetFigFont{9}{10.8}{\familydefault}{\mddefault}{\updefault}{\color[rgb]{0,0,0}$\overline{2}$}%
}}}}
\put(13501,-5956){\makebox(0,0)[b]{\smash{{\SetFigFont{9}{10.8}{\familydefault}{\mddefault}{\updefault}{\color[rgb]{0,0,0}$2$}%
}}}}
\put(12601,-5956){\makebox(0,0)[b]{\smash{{\SetFigFont{9}{10.8}{\familydefault}{\mddefault}{\updefault}{\color[rgb]{0,0,0}$2 \cdot 1$}%
}}}}
\put(12137,-1449){\makebox(0,0)[b]{\smash{{\SetFigFont{9}{10.8}{\familydefault}{\mddefault}{\updefault}{\color[rgb]{0,0,0}$\overline{\bot}$}%
}}}}
\put(10801,-2356){\makebox(0,0)[b]{\smash{{\SetFigFont{9}{10.8}{\familydefault}{\mddefault}{\updefault}{\color[rgb]{0,0,0}$\overline{1}$}%
}}}}
\put(11701,-2356){\makebox(0,0)[b]{\smash{{\SetFigFont{9}{10.8}{\familydefault}{\mddefault}{\updefault}{\color[rgb]{0,0,0}$\overline{1 \cdot 2}$}%
}}}}
\put(12601,-2356){\makebox(0,0)[b]{\smash{{\SetFigFont{9}{10.8}{\familydefault}{\mddefault}{\updefault}{\color[rgb]{0,0,0}$\overline{2 \cdot 1}$}%
}}}}
\put(11254,-3259){\makebox(0,0)[b]{\smash{{\SetFigFont{9}{10.8}{\familydefault}{\mddefault}{\updefault}{\color[rgb]{0,0,0}$\overline{1\circ 1 \cdot 2}$}%
}}}}
\put(13056,-3259){\makebox(0,0)[b]{\smash{{\SetFigFont{9}{10.8}{\familydefault}{\mddefault}{\updefault}{\color[rgb]{0,0,0}$\overline{2\circ 2 \cdot 1}$}%
}}}}
\put(11251,-4156){\makebox(0,0)[b]{\smash{{\SetFigFont{9}{10.8}{\familydefault}{\mddefault}{\updefault}{\color[rgb]{0,0,0}$1\diamond 1 \cdot 2$}%
}}}}
\put(12151,-4156){\makebox(0,0)[b]{\smash{{\SetFigFont{9}{10.8}{\familydefault}{\mddefault}{\updefault}{\color[rgb]{0,0,0}$1 \cdot 2\diamond 2 \cdot 1$}%
}}}}
\put(11251,-5056){\makebox(0,0)[b]{\smash{{\SetFigFont{9}{10.8}{\familydefault}{\mddefault}{\updefault}{\color[rgb]{0,0,0}$1\circ 1 \cdot 2$}%
}}}}
\put(13096,-5056){\makebox(0,0)[b]{\smash{{\SetFigFont{9}{10.8}{\familydefault}{\mddefault}{\updefault}{\color[rgb]{0,0,0}$2\circ 2 \cdot 1$}%
}}}}
\put(12151,-5056){\makebox(0,0)[b]{\smash{{\SetFigFont{9}{10.8}{\familydefault}{\mddefault}{\updefault}{\color[rgb]{0,0,0}$1 \cdot 2\circ 2 \cdot 1$}%
}}}}
\put(13051,-4156){\makebox(0,0)[b]{\smash{{\SetFigFont{9}{10.8}{\familydefault}{\mddefault}{\updefault}{\color[rgb]{0,0,0}$2\diamond 2 \cdot 1$}%
}}}}
\end{picture}%

\caption{$\mathbf{T}_2$ in Lemma~\ref{lemma:unifkcase2}.}\label{fig:posk2}
\end{figure}

We define $u_n \colon \mathbf{T}_n \to \mathbf{Q}$ as follows, where $j,k \in \{1,\dots,n\}$: 
% \begin{enumerate}
% \item[]  $u_n(\bot)=x$ and $u_n(i(\bot))=i(x)$;
% \item[] $u_n(j)=a$ and $u_n(i(j))=i(a)$ for all $j \in \{1,\dots,n\}$;
% \item[] $u_n(j\cdot k)=b$ and $u_n(i(j\cdot k))=i(b)$ for all $j<k$ in $\{1,\dots,n\}$;
% \item[] $u_n(j\cdot k)=c$ and $u_n(i(j\cdot k))=i(c)$ for all $k<j$ in $\{1,\dots,n\}$;
% \item[] $u_n(j \circ j\cdot k)=d$ and $u_n(i(j \circ j\cdot k))=i(d)$ for all $j<k$ in $\{1,\dots,n\}$;
% \item[] $u_n(j\cdot k \circ k\cdot j)=e$ and $u_n(i(j\cdot k \circ k\cdot j))=i(e)$ for all $j<k$ in $\{1,\dots,n\}$;
% \item[] $u_n(j \circ j\cdot k)=f$ and $u_n(i(j \circ j\cdot k))=i(f)$ for all $k<j$ in $\{1,\dots,n\}$;
% \item[] $u_n(j \diamond j\cdot k)=u_n(i(j \diamond j\cdot k))=y$ for all $j<k$ in $\{1,\dots,n\}$;
% \item[] $u_n(j\cdot k \diamond k\cdot j)=u_n(i(j\cdot k \diamond k\cdot j))=z$ for all $j<k$ in $\{1,\dots,n\}$;
% \item[] $u_n(j \diamond j\cdot k)=u_n(i(j \diamond j\cdot k))=w$ for all $k<j$ in $\{1,\dots,n\}$.
% \end{enumerate}
\begin{enumerate}
\item[]  $u_n(\bot)=x$;
\item[] $u_n(j)=a$ for all $j \in T_n$;
\item[] $u_n(j\cdot k)=b$ for all $j\cdot k \in T_n$ with $j<k$;
\item[] $u_n(j\cdot k)=c$ for all $j\cdot k \in T_n$ with $k<j$;
\item[] $u_n(j \circ j\cdot k)=d$ for all $j \circ j\cdot k \in T_n$ with $j<k$;
\item[] $u_n(j\cdot k \circ k\cdot j)=e$ for all $j\cdot k \circ k\cdot j \in T_n$ with $j<k$;
\item[] $u_n(j \circ j\cdot k)=f$ for all $j \circ j\cdot k \in T_n$ with $k<j$;
\item[] $u_n(j \diamond j\cdot k)=y$ for all $j \diamond j\cdot k \in T_n$ with $j<k$;
\item[] $u_n(j\cdot k \diamond k\cdot j)=z$ for all $j\cdot k \diamond k\cdot j \in T_n$ with $j<k$;
\item[] $u_n(j \diamond j\cdot k)=w$ for all $j \diamond j\cdot k \in T_n$ with $k<j$;
\end{enumerate}
and, for all $y \in L \cup I \subseteq T_n$, 
\begin{enumerate}
\item[] $u_n(i(y))=i(u_n(y))$.
\end{enumerate}
It is easy to check that $u_n \colon \mathbf{T}_n \to \mathbf{Q}$ is a unifier for $\mathbf{Q}$ in $V$.

Let $u \colon \mathbf{P} \to \mathbf{Q}$ be a unifier for $\mathbf{Q}$ in $V$,
where $\mathbf{P}=(P,\leq,i)$. We show that $u_n \leq u$ implies $|P| \geq n$.
Let $u_n=u \circ h$.  We claim that $h(j) \neq h(k)$ for all $j<k$ in $\{1,\dots,n\}$.  Let $j<k$ in $\{1,\dots,n\}$.
By construction,
$$j,j\cdot k \leq j \circ j\cdot k \leq i(j \circ j\cdot k) \leq i(j),i(j\cdot k)\text{,}$$
then
$$h(j),h(j\cdot k) \leq h(j \circ j\cdot k) \leq i(h(j \circ j\cdot k)) \leq i(h(j)),i(h(j\cdot k))\text{,}$$
and by $(K_1)$, $h(j) \vee h(j\cdot k)$ exists in $\mathbf{P}$ and
$$h(j) \vee h(j\cdot k) \leq i(h(j) \vee h(j\cdot k))\text{,}$$
so that $$h(j),h(j\cdot k) \leq h(j) \vee h(j\cdot k) \leq h(j \circ j\cdot k)\text{.}$$
Similarly,
$$h(j\cdot k),h(k\cdot j) \leq h(j\cdot k) \vee h(k\cdot j) \leq h(j\cdot k \circ k\cdot j)\text{,}$$
and
$$h(k\cdot j),h(k) \leq h(k\cdot j) \vee h(k) \leq h(k \circ k\cdot j)\text{.}$$
If we assume the contrary, that is, $h(j)=h(k)$, then
$$h(j) \vee h(l)=h(l) \vee h(k);$$
and applying $(M_3)$ to $h(j),h(j\cdot k),h(k\cdot j)$, we have
$$h(j),h(j\cdot k),h(k\cdot j) \leq h(j) \vee h(j\cdot k) \vee h(k\cdot j) \leq i(h(j) \vee h(j\cdot k) \vee h(k\cdot j))\text{.}$$
Applying $u$ through, recalling that $u_n=u \circ h$, we have
$$a,b,c \leq u(h(j) \vee h(j\cdot k) \vee h(k\cdot j)) \leq i(u(h(j) \vee h(j\cdot k) \vee h(k\cdot j)))\text{,}$$
which contradicts clause $(iii)$ in the statement. 

This proves that $\mathrm{type}(V)\neq 1$. Then $\mathrm{type}(V)=0$. Now, by Lemma~\ref{Lemma:UpsetNullary} $\mathrm{type}_{\K}(\mathbf{Q})=0$, as desired.
\end{proof}

The proof of the main result in this section (Theorem \ref{Theo:UnifClassKleene}) relies on the following notion.
\begin{definition}[Kleene Unification Core]\label{def:qprime} 
Let $\mathbf{Q}=(Q,\leq,i) \in \FPK$.  The \emph{Kleene unification core} of $\mathbf{Q}$ 
is the structure $\mathbf{Q'}=(Q',\leq',i') \in \FPK$ where:
\begin{enumerate}
\item[$(i)$] $Q'=\{x,i(x)\in Q \mid \text{$x\leq z=i(z)$ for some $z\in Q$}\}$;
\item[$(ii)$] $x \leq' y$ iff,
$x \leq y$ and either of the following three cases occurs:
\begin{enumerate}
\item[$(a)$] $x\leq i(x)$ and $y\leq i(y)$;
\item[$(b)$] $i(x)\leq x$ and $i(y)\leq y$;
\item[$(c)$] $x \leq z=i(z) \leq y$ for some $z\in Q$;
\end{enumerate}
\item[$(iii)$] $i'(x)=i(x)$ for all $x \in Q'$.
\end{enumerate}
\end{definition}
The following lemma justifies the terminology introduced. 
\begin{lemma}\label{Lem:UK}
Let $\mathbf{Q}=(Q,\leq,i) \in \FPK$ and $\mathbf{Q'} \in \FPK$ be its Kleene unification core.%  Then:
\begin{enumerate}
\item[$(i)$] If $u \colon \mathbf{P} \to \mathbf{Q}$ is a unifier for $\mathbf{Q}$,
then $u(P) \subseteq Q'$ and $u\colon \mathbf{P}\rightarrow \mathbf{Q'}$ is a unifier for $\mathbf{Q'}$.
\item[$(ii)$] $U_{\K}(\mathbf{Q}) \simeq U_{\K}(\mathbf{Q}')$.
\item[$(iii)$] $\mathbf{Q'}=(Q',\leq',i') \in \FPK$ satisfies $(M_2)$ and $(K_2)$.
\end{enumerate}
\end{lemma}
\begin{proof} 
$(i)$ Let $u \colon \mathbf{P} \to \mathbf{Q}$ in $\FPK$ be a unifier for $\mathbf{Q}$,
with $\mathbf{P}=(P,\leq_P,i_P)$.

We show that $u(P) \subseteq Q'$.  Let $x\in P$.  
Without loss of generality, we may assume $x\leq i_P(x)$. By $(M_2)$ there exists $z\in P$ such that $x\leq z=i_P(z)$. Then $u(x)\leq u(z)=i(u(z))$, concluding that $u(x)\in Q'$.
 
We show that $u\colon \mathbf{P}\rightarrow \mathbf{Q'}$ is a unifier for $\mathbf{Q'}$.  
For all $x \in P$, we have $u(i_P(x))=i(u(x))=i'(u(x))$ by part $(i)$ and Definition~\ref{def:qprime}$(iii)$.  
For monotonicity, let $x \leq_P y$.  If $u(x) \leq i(u(x))$ and $u(y) \leq i(u(y))$, 
or $i(u(x)) \leq u(x)$ and $i(u(y)) \leq u(y)$, then $u(x) \leq' u(y)$ by Definition~\ref{def:qprime}$(ii)$.  
Otherwise, assume that $u(x) \leq i(u(x))$, $i(u(y))\leq u(y)$, 
and there not exists $w \in Q'$ such that $u(x) \leq w=i(w) \leq u(y)$.  
Then, $u(x)<i(u(x))=u(i_P(x))$ and $u(i_P(y))=i(u(y))<u(y)$, 
which implies $x <_P i_P(x)$ and $i_P(y)<_P y$.  Since $x \leq_P y$ by hypothesis, 
we have $x,i_P(y) \leq_P i_P(x),y$.  By $(K_2)$, there exists $z \in P$ 
such that $x \leq_P z =_P i_P(z) \leq y$.  But then, 
$u(x) \leq u(z)=i(u(z)) \leq u(y)$, a contradiction.

$(ii)$ It follows from part $(i)$ and the fact that the inclusion map from $Q'$ into $Q$ 
is a morphism in $\FPK$. 
% Let $u \colon \mathbf{P} \to \mathbf{Q}$ be a unifier for $\mathbf{Q}$.  Then 
% by part $(i)$, the restriction of $u$ to $Q'$ in the codomain is a unifier for $\mathbf{Q'}$.  
% Conversely, if $u \colon \mathbf{P} \to \mathbf{Q}'$ is a unifier for $\mathbf{Q}'$, 
% then by composing with the inclusion map we obtain the inverse of unifier for $\mathbf{Q}$.  

$(iii)$ The statement holds by Definition~\ref{def:qprime}.  In details,
if $x \leq' i'(x)=i(x)$, either $x=i(x)$ or there exists $y\in Q$ such that $x\leq y=i(y)\leq i(x)$, 
so that $(M_2)$ holds in $\mathbf{Q'}$.  For $(K_2)$,
let $x,y \in Q'$ be such that $x,y \leq' i'(y),i'(x)$.
We want to show that there exists $z \in Q'$ such that $x,y \leq' z \leq i'(z)$.
Since $i'$ is the restriction of $i$ to $Q'$, $x,y \leq' i(y),i(x)$.
If $x= i(x)$ or $y=i(y)$ the result follows straightforwardly.
If $x< i(x)$ and $y< i(y)$, since $x\leq' i(y)$ by $(c)$ there exists
$z \in Q$ such that $x\leq z =i(z)\leq i(y)$.
Then $x,y \leq' z =i'(z)$, and the result follows.
\end{proof}

\begin{theorem}\label{Theo:UnifClassKleene}
Let $\mathbf{Q}=(Q,\leq,i) \in \FPK$ be a solvable instance of $\textsc{Unif}(\K)$ and $\mathbf{Q'} \in \FPK$ be the Kleene unification core of $\mathbf{Q}$.  Then:
\begin{align*}
\mathrm{type}_{\K}(\mathbf{Q}) &=
\begin{cases}
1\text{,} & \text{iff } \mathbf{Q'} \text{ satisfies } (K_1) \text{ and } (M_3)\\
\omega\text{,} & \text{iff $\mathbf{Q}'$ does not satisfy $(K_1$)}\\
               & \text{but $[x,i(x)]_{Q'}$ satisfies $(K_1)$ and $(M_3)$ for each $x\in Q'$;}\\
0\text{,} & \text{otherwise.}
\end{cases}
\end{align*}
\end{theorem}
\begin{proof}
Assume first that $\mathbf{Q}'$ satisfies $(K_1)$ and $(M_3)$. By Lemma~\ref{Lem:UK}$(iii)$, $\mathbf{Q}'$ satisfies $(M_2)$ and $(K_2)$.  
Then $\DK(\mathbf{Q}')$ is projective by Theorem \ref{th:kmain} and $\mathrm{type}_{\K}(\mathbf{Q'})=1$.  Now by Lemma~\ref{Lem:UK}$(ii)$  
$\mathrm{type}_{\K}(\mathbf{Q})=\mathrm{type}_{\K}(\mathbf{Q'})=1$.  

Suppose that $\mathbf{Q}'$ does not satisfy $(K_1)$ and $[x,i(x)]_{Q'}$ satisfies $(K_1)$ and $(M_3)$ for all $x \leq i(x)$ in $\mathbf{Q}'$.
Since $\mathbf{Q}'$ satisfies $(M_2)$ and $(K_2)$, it follows that $[x,i(x)]_{Q'}$ satisfies $(M_2)$ and $(K_2)$ for all $x \leq i(x)$ in $\mathbf{Q}'$.
Thus, define for every $x \in \mathrm{min}(\mathbf{Q}')$ the (inclusion) unifier $u_{x} \colon [x,i(x)]_{Q'} \to \mathbf{Q}'$
by $u_{x}(z)=z$ for all $z \in [x,i(x)]_{Q'}$.  Clearly, there are finitely many unifiers of the form $u_{x}$ in $\mathbf{Q}'$, 
because $Q'$ is finite, and at least one such unifier because $Q'$ is nonempty.  
We claim that the above unifiers form $\mu$-set in $U_{\K}(\mathbf{Q}')$.
Clearly, if $x \neq y$ are minimal points in $Q'$ then $u_{x}\parallel u_{y}$. Now, 
let $u \colon \mathbf{P} \to \mathbf{Q}'$ be a unifier for $\mathbf{Q}'$.
Now, $\mathbf{P}$ is bounded, with bottom $\bot$ let $x\in {\rm min}(\mathbf{Q}')$ such that $x\leq u(\bot)$ then
so that $u(P) \subseteq [x,i(x)]_{Q'}$ via the inclusion (monotone) map $f$,
but then $u_{x} \circ f=u$ so that $u_{x}$ is more general than
$u$.   Thus, $\mathrm{type}_{\K}(\mathbf{Q}') \in \{1,\omega\}$.  We claim that $\mathrm{type}_{\K}(\mathbf{Q}') \neq 1$,
that is, there exist two distinct unifiers for $\mathbf{Q}'$ with no common upper bound in $U_{\K}(\mathbf{Q}')$.  In fact,
$\mathbf{Q}'$ is not bounded, otherwise if $\mathbf{Q}'$ is bounded by  $\bot$ and $\top$,
then $\bot \leq i(\bot)=\top$ and then $[\bot,i(\bot)]_{Q'}=Q'$ satisfies $(K_1)$.
Hence, there are two distinct minimal points in $\mathbf{Q}'$, $x_1 \neq x_2$, so that $u_{x_1} \neq u_{x_2}$ are two distinct maximals in $U_{\K}(\mathbf{Q}')$.
Finally by Lemma \ref{Lem:UK}$(ii)$, $\mathrm{type}_{\K}(\mathbf{Q})=\mathrm{type}_{\K}(\mathbf{Q}')=\omega$.

Suppose now that there exists $x \leq i(x)$ in $Q'$  such that $[x,i(x)]_{Q'}$ does not satisfy $(K_1)$.
Then $\{z\in Q' \mid x\leq' z\leq' i(z)\}$ with restricted order is not a meet semilattice,
that is, there exist $x,a,b,c,d,y,z \in Q'$ such that $x \leq a,b \leq c,d$, $c \leq y=i(y)$, $d \leq z=i(z)$ but 
there does not exist $e\in Q'$ such that $a,b \leq e \leq c,d$.  By Lemma~\ref{lemma:unifkcase1}, $\mathrm{type}_{\K}(\mathbf{Q}')=0$.  
Thus, by Lemma \ref{Lem:UK}$(ii)$, $\mathrm{type}_{\K}(\mathbf{Q})=0$.

Finally suppose that for all $x \leq i(x)$ in $\mathbf{Q}'$ the interval $[x,i(x)]_{Q'}$ satisfies $(K_1)$, 
but there exists $x \leq i(x)$ in $\mathbf{Q}'$ such that $[x,i(x)]_{Q'}$ does not satisfy $(M_3)$; 
this case includes the case where $\mathbf{Q}'$ satisfies $(K_1)$ and does not satisfy $(M_3)$. 
Then there exist $x,a,b,c,d,e,f,y,z,w \in Q'$ such that: $x \leq a,b,c$; 
$a \leq d,e$; $b \leq d,f$; $c \leq e,f$; $d \leq y=i(y)$; $e \leq z=i(z)$; $f \leq w=i(w)$; 
there does not exist $g \in Q'$ such that $a,b,c \leq g \leq i(g)$.  By Lemma~\ref{lemma:unifkcase2}, $\mathrm{type}_{\K}(\mathbf{Q}')=0$.  
Finally by Lemma \ref{Lem:UK}$(ii)$, $\mathrm{type}_{\K}(\mathbf{Q})=0$.
\end{proof}

Using Lemma~\ref{lemma:unifkcase1} and Lemma~\ref{lemma:unifkcase2}, 
it is easy to construct examples of Kleene unification problems having nullary type, 
which proves that the variety of Kleene algebras has nullary equational unification type. 

\subsection{De Morgan Algebras}\label{subsect:unifdm}

We provide a complete classification of solvable instances of the unification problem
over De Morgan algebras (Theorem \ref{Theo:UnifClassDeMorgan}).
Using \cite{G97}, Theorem~\ref{th:dualities}, and Theorem~\ref{th:mmain}, 
the problem reduces to the following:

\begin{description}
\item[Problem]  $\textsc{Unif}(\M)$.
\item[Instance]  $\mathbf{Q}=(Q,\leq,i) \in \FPM$.
\item[Solution]  A morphism $u \colon \mathbf{P} \to \mathbf{Q}$ in $\FPM$,
where $\mathbf{P}$ satisfies $(M_1)$-$(M_3)$.
\end{description}

 This section follows a similar structure than the previous on. We first identify three sufficient conditions for an instance of the unification problem
to have nullary type (Lemma~\ref{lemma:unifMcase1}, Lemma~\ref{lemma:unifMcase2}, and Lemma~\ref{lemma:unifMcase3}),
and then we prove that the identified conditions are indeed necessary for nullarity (Theorem~\ref{Theo:UnifClassDeMorgan}).

\begin{remark}
An instance $\mathbf{Q}=(Q,\leq,i)$ of $\textsc{Unif}(\M)$ is solvable iff
$\{ x \in Q \mid x=i(x) \}\neq \emptyset$. The proof follows as in Remark \ref{Rem:KleeneSolvable}.
\end{remark}

Given a solvable instance $\mathbf{Q}$ of $\textsc{Unif}(\M)$, we let $U_{\M}(\mathbf{Q})$ 
denote the preordered set of unifiers of $\mathbf{Q}$, which is defined as in Section~\ref{subsect:unifbdl}.  

\begin{lemma}\label{Lem:UpDeMorgan}
Let $\mathbf{Q}=(Q,\leq,i) \in \FPM$ be an instance of $\textsc{Unif}(\M)$ and $x\in Q$ be a minimal element of $\mathbf{Q}$.
Then
\begin{equation}\label{Eq:UpwardDeMorgan}
V=\{ u \colon \mathbf{P} \rightarrow \mathbf{Q} \in U_{\M}(\mathbf{Q}) \mid x \in u(P)\}
\end{equation}
is a directed upset in $U_{\M}(\mathbf{Q})$.
\end{lemma}
\begin{proof}
Along the lines of Lemma~\ref{Lem:UpKleene}.
\end{proof}

\begin{lemma}\label{lemma:unifMcase1}
Let $\mathbf{Q}=(Q,\leq,i) \in \FPM$ be an instance of $\textsc{Unif}(\M)$.
If there exist $x,a,b,c,d,y \in Q$ such that:
\begin{itemize}
\item[$(i)$] $x \leq a,b \leq c,d$; 
\item[$(ii)$] $x \leq y=i(y)$;
\item[$(iii)$] there does not exist $e \in Q$ such that $a,b \leq e \leq c,d$;
\end{itemize}
then $\mathrm{type}_{\M}(\mathbf{Q})=0$ (see Figure~\ref{fig:mconf1}).
\begin{figure}
\centering
\begin{picture}(0,0)%
\includegraphics{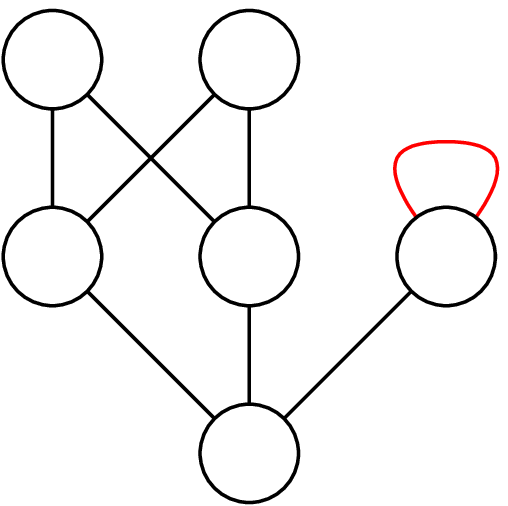}%
\end{picture}%
\setlength{\unitlength}{4144sp}%
\begingroup\makeatletter\ifx\SetFigFont\undefined%
\gdef\SetFigFont#1#2#3#4#5{%
  \reset@font\fontsize{#1}{#2pt}%
  \fontfamily{#3}\fontseries{#4}\fontshape{#5}%
  \selectfont}%
\fi\endgroup%
\begin{picture}(2297,2280)(6286,-5701)
\put(7430,-3714){\makebox(0,0)[b]{\smash{{\SetFigFont{9}{10.8}{\familydefault}{\mddefault}{\updefault}{\color[rgb]{0,0,0}$d$}%
}}}}
\put(6530,-3713){\makebox(0,0)[b]{\smash{{\SetFigFont{9}{10.8}{\familydefault}{\mddefault}{\updefault}{\color[rgb]{0,0,0}$c$}%
}}}}
\put(7434,-4613){\makebox(0,0)[b]{\smash{{\SetFigFont{9}{10.8}{\familydefault}{\mddefault}{\updefault}{\color[rgb]{0,0,0}$b$}%
}}}}
\put(6529,-4604){\makebox(0,0)[b]{\smash{{\SetFigFont{9}{10.8}{\familydefault}{\mddefault}{\updefault}{\color[rgb]{0,0,0}$a$}%
}}}}
\put(7427,-5508){\makebox(0,0)[b]{\smash{{\SetFigFont{9}{10.8}{\familydefault}{\mddefault}{\updefault}{\color[rgb]{0,0,0}$x$}%
}}}}
\put(8327,-4608){\makebox(0,0)[b]{\smash{{\SetFigFont{9}{10.8}{\familydefault}{\mddefault}{\updefault}{\color[rgb]{0,0,0}$y$}%
}}}}
\end{picture}%

\caption{Subposet of $\mathbf{Q}$ in Lemma~\ref{lemma:unifMcase1}.}\label{fig:mconf1}
\end{figure}
\end{lemma}
\begin{proof}
Since $\mathbf{Q}$ is a finite poset, we assume without loss of generality $x \in \mathrm{min}(\mathbf{Q})$.
Notice that by $(iii)$, we have $a\neq b$ and $c\neq d$.  Let,
$$V=\{ u \colon \mathbf{P} \rightarrow \mathbf{Q} \in U_{\M}(\mathbf{Q}) \mid x \in u(P)\}\text{.}$$
By Lemma \ref{Lem:UpDeMorgan}, $V$ is an directed upset in $U_{\M}(\mathbf{Q})$.  
By Lemma \ref{Lemma:UpsetNullary},  to prove that ${\rm type}(U_{\M}(\mathbf{Q}))=0$ it is enough to prove that ${\rm type}(V)=0$.  
Since $V$ is directed, by Lemma \ref{Lemma:DirectedType}, ${\rm type}(V)\in\{0,1\}$.
We show that $\mathrm{type}(V) \neq 1$.  

For every $n \in \mathbb{N}$, we define a unifier $u_n \colon \mathbf{T}_n \to \mathbf{Q}$ in $V$ as follows.
For $\mathbf{T}_n=(T_n,\leq,i) \in \FPK$ we let
$$T_n=\{ \bot,\overline{\bot},0,j,\overline{j},j\cdot k,\overline{j\cdot k} \mid \text{$j<k$ in $\{1,\dots,n\}$ and $j + k$ is odd} \}\text{.}$$
The map $i \colon T_n \to T_n$ is defined by: 
\begin{enumerate}
\item[] $i(0)=0$;
\item[] $i(y)=\overline{y}$ and $i(\overline{y})=y$ for all $y\in T_n\setminus \{0\}$.
\end{enumerate}
The partial order over $T_n$ is defined by the following cover relation, for all $j,k \in \{1,\dots,n\}$:
\begin{enumerate}
\item[] $\bot \prec j$ and $i(j) \prec i(\bot)$;
\item[] $\bot \prec i(j\cdot k)$ and $j\cdot k \prec i(\bot)$ if $j<k$;
\item[] $j \prec j\cdot k$ and $i(j\cdot k) \prec i(j)$ if $j<k$;
\item[] $j \prec k\cdot j$ and $i(k\cdot j) \prec i(j)$ if $k<j$.
\end{enumerate}
It is easy to check that $\mathbf{T}_n$ satisfies $(M_1)$-$(M_3)$.
Figure~\ref{fig:posm1} provides the Hasse diagram of $\mathbf{T}_3$.
\begin{figure}
\centering
\begin{picture}(0,0)%
\includegraphics{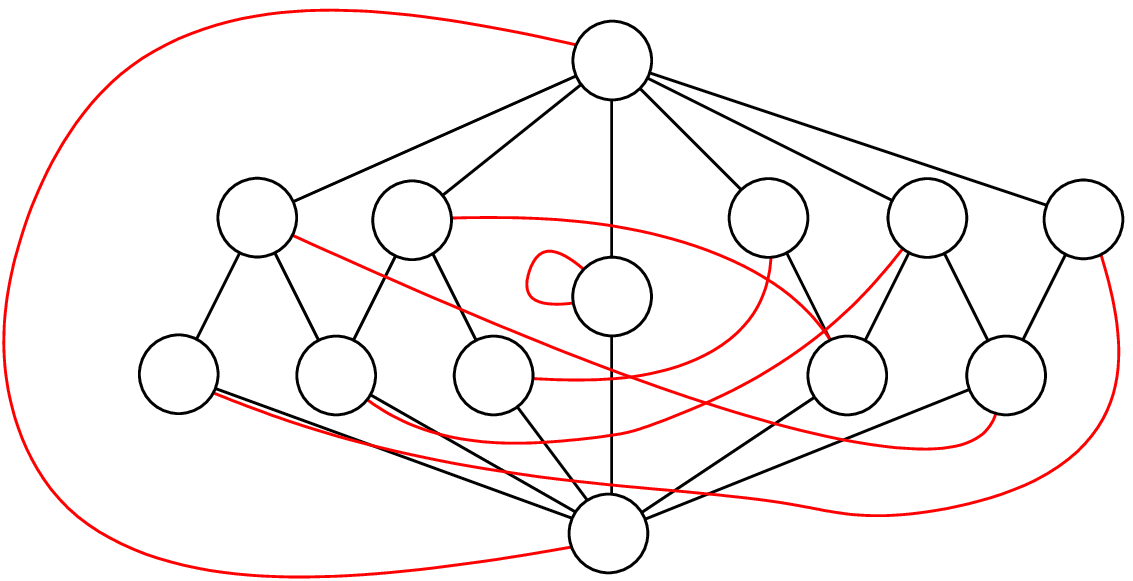}%
\end{picture}%
\setlength{\unitlength}{3315sp}%
\begingroup\makeatletter\ifx\SetFigFont\undefined%
\gdef\SetFigFont#1#2#3#4#5{%
  \reset@font\fontsize{#1}{#2pt}%
  \fontfamily{#3}\fontseries{#4}\fontshape{#5}%
  \selectfont}%
\fi\endgroup%
\begin{picture}(6433,3284)(8655,-5737)
\put(9676,-4606){\makebox(0,0)[b]{\smash{{\SetFigFont{7}{8.4}{\familydefault}{\mddefault}{\updefault}{\color[rgb]{0,0,0}$1$}%
}}}}
\put(10576,-4606){\makebox(0,0)[b]{\smash{{\SetFigFont{7}{8.4}{\familydefault}{\mddefault}{\updefault}{\color[rgb]{0,0,0}$2$}%
}}}}
\put(11476,-4606){\makebox(0,0)[b]{\smash{{\SetFigFont{7}{8.4}{\familydefault}{\mddefault}{\updefault}{\color[rgb]{0,0,0}$3$}%
}}}}
\put(10126,-3706){\makebox(0,0)[b]{\smash{{\SetFigFont{7}{8.4}{\familydefault}{\mddefault}{\updefault}{\color[rgb]{0,0,0}$1 \cdot 2$}%
}}}}
\put(11026,-3706){\makebox(0,0)[b]{\smash{{\SetFigFont{7}{8.4}{\familydefault}{\mddefault}{\updefault}{\color[rgb]{0,0,0}$2 \cdot 3$}%
}}}}
\put(12151,-5506){\makebox(0,0)[b]{\smash{{\SetFigFont{7}{8.4}{\familydefault}{\mddefault}{\updefault}{\color[rgb]{0,0,0}$\bot$}%
}}}}
\put(12151,-4156){\makebox(0,0)[b]{\smash{{\SetFigFont{7}{8.4}{\familydefault}{\mddefault}{\updefault}{\color[rgb]{0,0,0}$0$}%
}}}}
\put(14401,-4606){\makebox(0,0)[b]{\smash{{\SetFigFont{7}{8.4}{\familydefault}{\mddefault}{\updefault}{\color[rgb]{0,0,0}$\overline{1 \cdot 2}$}%
}}}}
\put(13051,-3706){\makebox(0,0)[b]{\smash{{\SetFigFont{7}{8.4}{\familydefault}{\mddefault}{\updefault}{\color[rgb]{0,0,0}$\overline{3}$}%
}}}}
\put(13951,-3706){\makebox(0,0)[b]{\smash{{\SetFigFont{7}{8.4}{\familydefault}{\mddefault}{\updefault}{\color[rgb]{0,0,0}$\overline{2}$}%
}}}}
\put(14806,-3706){\makebox(0,0)[b]{\smash{{\SetFigFont{7}{8.4}{\familydefault}{\mddefault}{\updefault}{\color[rgb]{0,0,0}$\overline{1}$}%
}}}}
\put(13501,-4606){\makebox(0,0)[b]{\smash{{\SetFigFont{7}{8.4}{\familydefault}{\mddefault}{\updefault}{\color[rgb]{0,0,0}$\overline{2 \cdot 3}$}%
}}}}
\put(12151,-2806){\makebox(0,0)[b]{\smash{{\SetFigFont{7}{8.4}{\familydefault}{\mddefault}{\updefault}{\color[rgb]{0,0,0}$\overline{\bot}$}%
}}}}
\end{picture}%

\caption{$\mathbf{T}_3$ in Lemma~\ref{lemma:unifMcase1}.}\label{fig:posm1}
\end{figure}
We define $u_n \colon \mathbf{T}_n \to \mathbf{Q}$ as follows, where $j,k \in \{1,\dots,n\}$: 
\begin{enumerate}
\item[] $u_n(\bot)=x$;
\item[] $u_n(j)=a$ and $u_n(j\cdot k)=c$, for all $j,j\cdot k \in T_n$ with $j$ odd; 
\item[] $u_n(j)=b$ and $u_n(j\cdot k)=d$, for all $j,j\cdot k \in T_n$ with $j$ even;
\item[] $u_n(0)=y$;
\end{enumerate}
and, for all $y \in \{ \bot,0,j,j\cdot k \mid \text{$j<k$ in $\{1,\dots,n\}$ and $j + k$ is odd} \} \subseteq T_n$, 
\begin{enumerate}
\item[] $u_n(i(y))=i(u_n(y))$.
\end{enumerate}
It is easy to check that $u_n \colon \mathbf{T}_n \to \mathbf{Q}$ is a unifier for $\mathbf{Q}$ in $V$.

Let $u \colon \mathbf{P} \to \mathbf{Q}$ be a unifier for $\mathbf{Q}$ such that $u \in V$.  We show that $u_n \leq u$ implies $|P| \geq n$.
Let $u_n=u \circ f$.  We claim that $f(j) \neq f(k)$ for all $j<k$ with $j,k \in \{1,\dots,n\}$.
If $j$ and $k$ have a different parity, then it is clear.   
If $j$ and $k$ have the same parity, 
without loss of generality assume that both are odd, then let $l$ be even such that $j<l<k$.  
Since by construction $j,l \leq j\cdot l$,
we have $f(j),f(l) \leq f(j\cdot l)$.  By $(M_1)$,
$$f(j),f(l) \leq f(j) \vee f(l) \leq f(j\cdot l)\text{.}$$
Similarly,
$$f(l),f(k) \leq f(l) \vee f(k) \leq f(l\cdot k)\text{.}$$
Assume $f(j)=f(k)$ for a contradition.  Then
$$f(j)=f(k),f(l) \leq f(j) \vee f(l)=f(l) \vee f(k) \leq f(j\cdot l),f(l\cdot k)\text{,}$$
and applying $u$ through recalling that $u_n=u \circ f$,
$$a,b \leq u(f(l) \vee f(k)) \leq c,d\text{,}$$
contradicting clause $(iii)$ in the statement. 

This proves that $\mathrm{type}(V)\neq 1$.  Then $\mathrm{type}(V)=0$.  
Therefore, $\mathrm{type}_{\M}(\mathbf{Q})=0$, as desired.
\end{proof}

\begin{lemma}\label{lemma:unifMcase2}
Let $\mathbf{Q}=(Q,\leq,i) \in \FPM$ be an instance of $\textsc{Unif}(\M)$.
If there exist $x,a,b \in Q$ such that:
\begin{itemize}
\item[$(i)$] $x \leq a,b$;
\item[$(ii)$] $a \leq i(a)$; $b=i(b)$;
\item[$(iii)$] there does not exist $c \in Q$ such that $a \leq c=i(c)$;
\end{itemize}
then $\mathrm{type}_{\M}(\mathbf{Q})=0$ (see Figure~\ref{fig:mconf2}).
\begin{figure}
\centering
\begin{picture}(0,0)%
\includegraphics{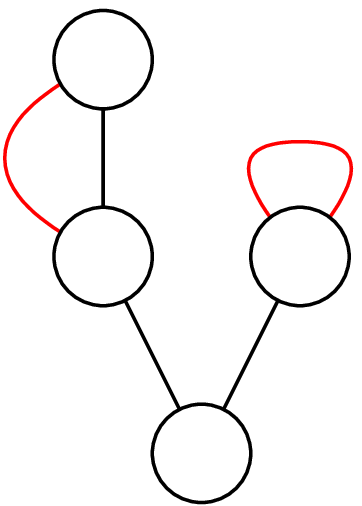}%
\end{picture}%
\setlength{\unitlength}{4144sp}%
\begingroup\makeatletter\ifx\SetFigFont\undefined%
\gdef\SetFigFont#1#2#3#4#5{%
  \reset@font\fontsize{#1}{#2pt}%
  \fontfamily{#3}\fontseries{#4}\fontshape{#5}%
  \selectfont}%
\fi\endgroup%
\begin{picture}(1629,2280)(6954,-5701)
\put(7430,-3714){\makebox(0,0)[b]{\smash{{\SetFigFont{9}{10.8}{\familydefault}{\mddefault}{\updefault}{\color[rgb]{0,0,0}$i(a)$}%
}}}}
\put(7434,-4613){\makebox(0,0)[b]{\smash{{\SetFigFont{9}{10.8}{\familydefault}{\mddefault}{\updefault}{\color[rgb]{0,0,0}$a$}%
}}}}
\put(8327,-4608){\makebox(0,0)[b]{\smash{{\SetFigFont{9}{10.8}{\familydefault}{\mddefault}{\updefault}{\color[rgb]{0,0,0}$b$}%
}}}}
\put(7877,-5508){\makebox(0,0)[b]{\smash{{\SetFigFont{9}{10.8}{\familydefault}{\mddefault}{\updefault}{\color[rgb]{0,0,0}$x$}%
}}}}
\end{picture}%

\caption{Subposet of $\mathbf{Q}$ in Lemma~\ref{lemma:unifMcase2}.}\label{fig:mconf2}
\end{figure}
\end{lemma}
\begin{proof}
Since $\mathbf{Q}$ is a finite poset, we assume without loss of generality $x \in \mathrm{min}(\mathbf{Q})$.  
Let,
$$V=\{ u \colon \mathbf{P} \rightarrow \mathbf{Q} \in U_{\M}(\mathbf{Q}) \mid x \in u(P)\}\text{.}$$
By Lemma \ref{Lem:UpDeMorgan}, $V$ is an directed upset in $U_{\M}(\mathbf{Q})$.  
By Lemma \ref{Lemma:UpsetNullary},  to prove that ${\rm type}(U_{\K}(\mathbf{Q}))=0$ it is enough to prove that ${\rm type}(V)=0$.  
Since $V$ is directed, by Lemma \ref{Lemma:DirectedType}, ${\rm type}(V)\in\{0,1\}$.
We show that $\mathrm{type}(V) \neq 1$.  

For every odd $n \in \mathbb{N}$, let
\begin{align*}
T_n &= \{ 0,1\}^n \cup \{d\}\text{.}
\end{align*}
The map $i \colon T_n \to T_n$ is defined by: 
\begin{enumerate}
\item[] $i(d)=d$;
\item[] $i(e_1,\ldots, e_n)=(f_1,\ldots, f_n)$ where $f_j=0$ iff $e_j=1$ for $j=1,2,\dots,n$.
\end{enumerate}
The partial order over $T_n$ is defined by: 
\begin{enumerate}
\item[] $(e_1,\ldots, e_n)\leq(f_1,\ldots, f_n)$ if $e_j \leq f_j$ for $j=1,2,\dots,n$; 
\item[] $(0,\ldots,0)\leq d\leq (1,\ldots,1)$.
\end{enumerate}
It is easy to check that $\mathbf{T}_n$ satisfies $(M_1)$, $(M_2)$, and $(M_3)$.
Figure~\ref{fig:posdmp2} provides the Hasse diagram of $\mathbf{T}_3$.
\begin{figure}
\centering
\begin{picture}(0,0)%
\includegraphics{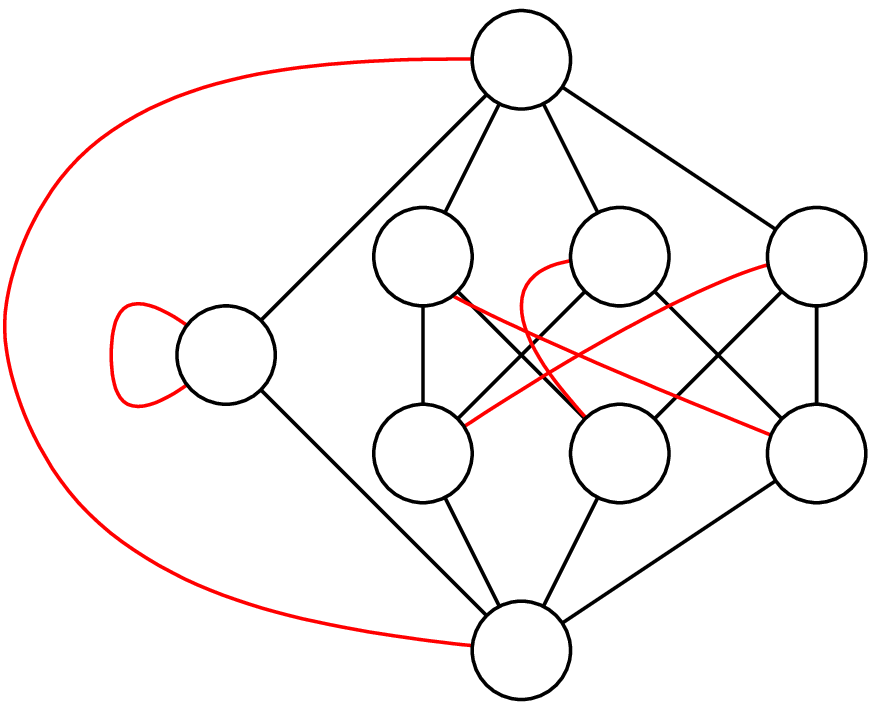}%
\end{picture}%
\setlength{\unitlength}{4144sp}%
\begingroup\makeatletter\ifx\SetFigFont\undefined%
\gdef\SetFigFont#1#2#3#4#5{%
  \reset@font\fontsize{#1}{#2pt}%
  \fontfamily{#3}\fontseries{#4}\fontshape{#5}%
  \selectfont}%
\fi\endgroup%
\begin{picture}(3974,3180)(4592,-5701)
\put(8326,-3661){\makebox(0,0)[b]{\smash{{\SetFigFont{6}{7.2}{\familydefault}{\mddefault}{\updefault}{\color[rgb]{0,0,0}$(1,1,0)$}%
}}}}
\put(7426,-3661){\makebox(0,0)[b]{\smash{{\SetFigFont{6}{7.2}{\familydefault}{\mddefault}{\updefault}{\color[rgb]{0,0,0}$(1,0,1)$}%
}}}}
\put(6526,-3661){\makebox(0,0)[b]{\smash{{\SetFigFont{6}{7.2}{\familydefault}{\mddefault}{\updefault}{\color[rgb]{0,0,0}$(0,1,1)$}%
}}}}
\put(6976,-2761){\makebox(0,0)[b]{\smash{{\SetFigFont{6}{7.2}{\familydefault}{\mddefault}{\updefault}{\color[rgb]{0,0,0}$(1,1,1)$}%
}}}}
\put(6526,-4561){\makebox(0,0)[b]{\smash{{\SetFigFont{6}{7.2}{\familydefault}{\mddefault}{\updefault}{\color[rgb]{0,0,0}$(0,0,1)$}%
}}}}
\put(7426,-4561){\makebox(0,0)[b]{\smash{{\SetFigFont{6}{7.2}{\familydefault}{\mddefault}{\updefault}{\color[rgb]{0,0,0}$(0,1,0)$}%
}}}}
\put(8326,-4561){\makebox(0,0)[b]{\smash{{\SetFigFont{6}{7.2}{\familydefault}{\mddefault}{\updefault}{\color[rgb]{0,0,0}$(1,0,0)$}%
}}}}
\put(6976,-5461){\makebox(0,0)[b]{\smash{{\SetFigFont{6}{7.2}{\familydefault}{\mddefault}{\updefault}{\color[rgb]{0,0,0}$(0,0,0)$}%
}}}}
\put(5626,-4111){\makebox(0,0)[b]{\smash{{\SetFigFont{6}{7.2}{\familydefault}{\mddefault}{\updefault}{\color[rgb]{0,0,0}$d$}%
}}}}
\end{picture}%

\caption{$\mathbf{T}_3$ in Lemma~\ref{lemma:unifMcase2}.}\label{fig:posdmp2}
\end{figure}

For each $n$ odd we define $u_n \colon \mathbf{T}_n \to \mathbf{Q}$ by: 
\begin{enumerate}
\item[] $u_n(0,\ldots,0)=x$ and $u_n(1,\ldots,1)=i(x)$;
\item[] $u_n(d)=b$;
\item[] $u_n(e_1,\ldots, e_n)=a$ if $1\leq e_1+\cdots+e_n < n/2$; 
\item[] $u_n(e_1,\ldots, e_n)=i(a)$ if $n/2< e_1 +\ldots + e_n\leq n-1$.
\end{enumerate}
It follows straightforwardly that $u_n \colon \mathbf{T}_n \to \mathbf{Q}$ is morphism in $\FPM$, and therefore a unifier for $\mathbf{Q}$ in $V$.

Let $u \colon \mathbf{P} \to \mathbf{Q}$ be a unifier for $\mathbf{Q}$ such that $u \in V$.  
We show that $u_n \leq u$ implies $|P| \geq n$.  Let $u_n=u \circ h$.  
We claim that $h(e) \neq h(f)$ for all $e \neq f$ in $T_n$ such that $e_1+\cdots+e_n =f_1+\cdots+f_n =1$.  Suppose 
for a contradiction that $h(e)=h(f)$.
By construction,  $e \leq i(f)$, 
therefore $$h(e) \leq h(i(f))=i(h(f))=i(h(e))\text{;}$$ 
since $\mathbf{P}$ satisfies $(M_2)$, there exists $z \in P$ 
such that $$h(e) \leq z=i(z) \leq i(h(e))\text{;}$$ 
but applying $u$:
$$a \leq u(z)=i(u(z)) \leq i(a)\text{,}$$
in clear contradiction with clause $(iii)$ in the statement. 

This proves that $\mathrm{type}(V)\neq 1$.  Then $\mathrm{type}(V)=0$.  
Therefore, $\mathrm{type}_{\M}(\mathbf{Q})=0$, as desired.
\end{proof}

\begin{lemma}\label{lemma:unifMcase3}
Let $\mathbf{Q}=(Q,\leq,i) \in \FPM$ be an instance of $\textsc{Unif}(\M)$.
If there exist $x,a,b,c,d,e,f,y,z,w \in Q$ such that:
\begin{itemize}
\item[$(i)$] $x\leq a,b,c$; $a\leq d,e$; $b\leq d,f$; $c\leq e,f$; 
\item[$(ii)$] $d\leq y=i(y)$; $e\leq z=i(z)$; $f\leq w=i(w)$; 
\item[$(iii)$] there does not exist $g\in Q$ such that $a,b,c \leq g\leq i(g)$;
\end{itemize}
then $\mathrm{type}_{\M}(\mathbf{Q})=0$ (see Figure~\ref{fig:kconf2}).
\end{lemma}
\begin{proof}
Observe that conditions $(i)$-$(iii)$ above are exactly conditions $(i)$-$(iii)$ in Lemma \ref{lemma:unifkcase2}.  
Moreover, as observed in Lemma \ref{lemma:unifkcase2}, the structure $\mathbf{T}_n$ satisfies $(M_1)$-$(M_3)$.  
Therefore, $u_n$ is a unifier for $\mathbf{Q}$ in $\FPM$ for all $n \in \mathbb{N}$, 
and the argument in Lemma~\ref{lemma:unifkcase2} applies.
\end{proof}

\begin{definition}[De Morgan Unification Core]\label{def:dmcore}
Let $\mathbf{Q}=(Q,\leq,i) \in \FPM$.
The \emph{De Morgan unification core} of $\mathbf{Q}$ in $\FPM$
is the structure $\mathbf{Q'}=(Q',\leq',i') \in \FPM$ defined by:
\begin{enumerate}
\item[$(i)$] $Q'=\{x,i(x)\in Q \mid \text{$y\leq z,x,i(x)$ for some $y,z\in Q$ such that $z=i(z)$}\}$;
\item[$(ii)$] $x \leq' y$  iff $x \leq y$;
\item[$(iii)$] $i'(x)=i(x)$ for all $x \in Q'$.
\end{enumerate}
\end{definition}

\begin{lemma}\label{lemma:dmcore}
Let $\mathbf{Q'}=(Q',\leq',i')$ be the De Morgan unification core of $\mathbf{Q}=(Q,\leq,i)$.  Then:
\begin{enumerate}
\item[$(i)$] If $u \colon \mathbf{P} \to \mathbf{Q}$ is a unifier for $\mathbf{Q}$,
then $u(P) \subseteq Q'$ and $u\colon \mathbf{P}\rightarrow \mathbf{Q'}$ is a unifier for $\mathbf{Q'}$.
\item[$(ii)$] $U_{\M}(\mathbf{Q}) \simeq U_{\M}(\mathbf{Q}')$.
\end{enumerate}
\end{lemma}
\begin{proof}
$(i)$ We claim that $u(P) \subseteq Q'$.  Indeed, let $x \in P$.  If $x \leq_P i_P(x)$,
then by $(M_2)$ there exists $z \in P$ such that $z=i_P(z)$ and $x \leq_P z$.
Then $u(x) \leq u(z)=i(u(z)) \leq i(u(x))$, so that $u(x) \in Q'$ by Definition~\ref{def:dmcore}$(i)$.
If $x \parallel_P i_P(x)$, then by $(M_1)$, there exists $x\wedge i(x)$ and it satisfies 
$x\wedge i(x)\leq i(x\wedge i(x))$. By $(M_2)$, there exists $z \in P$ such that $z=i_P(z)$
and $x\wedge i(x) \leq_P x,i_P(x),z$.  Then $u(x\wedge i(x)) \leq u(x),i(u(x)),u(z)$,
so that $u(x) \in Q'$ by Definition~\ref{def:dmcore}$(i)$.

$(ii)$ It follows from part $(i)$ and the fact that the inclusion $Q' \subseteq Q$ is in $\FPM$. 
\end{proof}

\begin{theorem}\label{Theo:UnifClassDeMorgan}
Let $\mathbf{Q}=(Q,\leq,i) \in \FPM$ be a solvable instance of $\textsc{Unif}(\M)$,
and $\mathbf{Q'}=(Q',\leq',i') \in \FPM$ be the De Morgan unification core of $\mathbf{Q}$.  Then:
\begin{align*}
\mathrm{type}_{\M}(\mathbf{Q}) &=
\begin{cases}
1\text{,} & \text{iff $\mathbf{Q'}$ satisfies $(M_1)$, $(M_2)$, and $(M_3)$}\\
\omega\text{,} & \text{iff $\mathbf{Q}'$ does not satisfy $(M_1)$, but for every $x \in Q'$} \\
               & \text{ with $x \leq' i(x)$, $[x,i(x)]_{Q'}$ satisfies $(M_1)$, $(M_2)$, and $(M_3)$;}\\
0\text{,} & \text{otherwise.}
\end{cases}
\end{align*}
\end{theorem}
\begin{proof}
If $\mathbf{Q}'$ satisfies $(M_1)$-$(M_3)$, $\DM(\mathbf{Q}')$ is projective by Theorem~\ref{th:mmain}, 
and $\mathrm{type}_{\M}(\mathbf{Q}')=1$.  Therefore, $\mathrm{type}_{\M}(\mathbf{Q})=1$
because $\mathrm{type}_{\M}(\mathbf{Q})=\mathrm{type}_{\M}(\mathbf{Q}')$ by Lemma~\ref{lemma:dmcore}$(ii)$.

Suppose that $\mathbf{Q}'$ does not satisfy $(M_1)$ and $[x,i(x)]_{Q'}$ satisfies $(M_1)$, $(M_2)$, and $(M_3)$
for all $x \leq' i(x)$ in $\mathbf{Q}'$.  Along the lines of the second part of the proof of Theorem~\ref{Theo:UnifClassKleene}, 
it follows that $\mathrm{type}_{\M}(\mathbf{Q})=\omega$.  
% Thus, define for every $x \in {\rm min}(\mathbf{Q}')$
% the (inclusion) unifier $u_{x} \colon [x,i(x)]_{Q'} \to \mathbf{Q}'$
% by $u_{x}(z)=z$ for all $z \in [x,i(x)]_{Q'}$.  Clearly, there are finitely many unifiers of the form $u_{x}$
% in $\mathbf{Q}'$, because $Q'$ is finite, and at least one such unifier because $Q'$ is nonempty.
% We claim that the above unifiers form $\mu$-set in $U_{\M}(\mathbf{Q}')$.
% Clearly, if $x\neq y$ then $u_{x}\parallel u_{y}$. Now, let $u \colon \mathbf{P} \to \mathbf{Q}'$ be a unifier for $\mathbf{Q}'$.
% $\mathbf{P}$ is bounded, with bottom $\bot$.  Let $x\in {\rm min}(\mathbf{Q}')$
% such that $x\leq u(\bot)$.  Then $u(P) \subseteq [x,i(x)]_{Q'}$ via the inclusion (monotone) map $f$,
% but then $u_{x} \circ f=u$ so that $u_{x}$ is more general than $u$.
% Thus, $\mathrm{type}_{\M}(\mathbf{Q}') \in \{1,\omega\}$.  We claim that $\mathrm{type}_{\M}(\mathbf{Q}') \neq 1$,
% that is, there exist two distinct unifiers for $\mathbf{Q}'$ with no common upper bound in $U_{\M}(\mathbf{Q}')$.  In fact,
% $\mathbf{Q}'$ is not bounded, otherwise if $\bot$ and $\top$ bound $\mathbf{Q}'$,
% then $\bot \leq i(\bot)=\top$ and then $[\bot,i(\bot)]_{Q'}=Q'$ satisfies $(M_1)$ and $(M_3)$.
% Hence, there are two distinct minimal points in $\mathbf{Q}'$,
% say that $x_1 \neq x_2$, so that $i(x_1) \neq i(x_2)$ are two distinct maximal points in $\mathbf{Q}'$,
% and the claim follows.
% Finally, $\mathrm{type}_{\M}(\mathbf{Q})=\mathrm{type}_{\M}(\mathbf{Q}')=\omega$ by Lemma~\ref{lemma:dmcore}$(ii)$.

Now suppose that there exist $x \leq' i(x)$ in $\mathbf{Q}'$ such that $[x,i(x)]_{Q'}$ does not satisfy $(M_1)$;
without loss of generality, $x \in {\rm min}(\mathbf{Q}')$.
Then $[x,i(x)]_{Q'}$ with restricted order is not a lattice,
that is, there exist $a,b,c,d \in Q'$ such that $x \leq' a,b \leq' c,d \leq i(x)$
but there does not exist $e \in Q'$ such that $a,b \leq' e \leq' c,d$.
Moreover, by minimality of $x$ and Definition~\ref{def:dmcore}$(i)$, there exists $y \in Q'$ such that $x \leq y=i(y)$.  
Therefore, by Lemma~\ref{lemma:unifMcase1}, $\mathrm{type}_{\M}(\mathbf{Q}')=0$. Thus,
by Lemma~\ref{lemma:dmcore}$(ii)$, $\mathrm{type}_{\M}(\mathbf{Q})=0$.

Next suppose that for all $x \leq' i(x)$ in $\mathbf{Q}'$ the interval $[x,i(x)]_{Q'}$ satisfies $(M_1)$, 
but there exists $x \leq' i(x)$ in $\mathbf{Q}'$ such that $[x,i(x)]_{Q'}$ does not satisfy $(M_2)$; 
without loss of generality, $x \in {\rm min}(\mathbf{Q}')$.  
This case includes the case where $\mathbf{Q}'$ satisfies $(M_1)$ and not $(M_2)$.  
Then there exists $a \leq i(a)$ in $[x,i(x)]_{Q'}$ such that there does not exist $c \in [x,i(x)]_{Q'}$ 
satisfying $a \leq c=i(c)$.  By minimality of $x$ and Definition~\ref{def:dmcore}$(i)$, 
there exists $b=i(b)$ in $Q'$ such that $x \leq b$.  Therefore by 
Lemma~\ref{lemma:unifMcase2}, $\mathrm{type}_{\M}(\mathbf{Q}')=0$. Thus,
by Lemma~\ref{lemma:dmcore}$(ii)$, $\mathrm{type}_{\M}(\mathbf{Q})=0$. 

Finally suppose that for all $x \leq' i(x)$ in $\mathbf{Q}'$ the interval $[x,i(x)]_{Q'}$ satisfies $(M_1)$ and $(M_2)$, 
but there exists $x \leq' i(x)$ in $\mathbf{Q}'$ such that $[x,i(x)]_{Q'}$ does not satisfy $(M_3)$; 
this case includes the case where $\mathbf{Q}'$ satisfies $(M_1)$ and $(M_2)$ but not $(M_3)$.  
Then there exist $x,a,b,c,d,e,f,y,z,w \in Q'$ such that: $x \leq' a,b,c$; 
$a \leq' d,e$; $b \leq' d,f$; $c \leq' e,f$; $d \leq' y=i(y)$; $e \leq' z=i(z)$; $f \leq' w=i(w)$; 
there does not exist $g \in Q'$ such that $a,b,c \leq' g \leq i(g)$.
Therefore, by Lemma~\ref{lemma:unifMcase3}, $\mathrm{type}_{\M}(\mathbf{Q}')=0$. Thus,
by Lemma~\ref{lemma:dmcore}$(ii)$, $\mathrm{type}_{\M}(\mathbf{Q})=0$.
\end{proof}

Lemma~\ref{lemma:unifMcase1}, Lemma~\ref{lemma:unifMcase2}, and Lemma~\ref{lemma:unifMcase3} 
yield examples of De Morgan unification instances having nullary type, 
proving that De Morgan algebras have nullary unification type.

%\bibliographystyle{plain}
%\bibliography{bova-cabrer-order12}

\begin{thebibliography}{10}

\bibitem{A87}
M.~E. Adams.
\newblock {Principal Congruences in De Morgan Algebras}.
\newblock {\em Proc. Edinburgh Math. Soc.}, 30:415--421, 1987.

\bibitem{BS01}
F.~Baader and W.~Snyder.
\newblock Unification theory.
\newblock In A.~Robinson and A.~Voronkov, editors, {\em Handbook of Automated
  Reasoning}, volume~8, pages 445--533. Elsevier, 2001.

\bibitem{BH70}
R.~Balbes and A.~Horn.
\newblock {Injective and Projective Heyting Algebras}.
\newblock {\em Trans. Amer. Math. Soc.}, 148:549--559, 1970.

\bibitem{BH70b}
R.~Balbes and A.~Horn.
\newblock {Projective distributive lattices}.
\newblock {\em Pacific J. Math.}, 33:273--279, 1970.

\bibitem{BB89}
J.~Berman and W.~J. Blok.
\newblock {Generalizations of Tarski's Fixed Point Theorem for Order Varieties
  of Complete Meet Semilattices}.
\newblock {\em Order}, 5:381--392, 1989.

\bibitem{B37}
G.~Birkhoff.
\newblock {Rings of Sets}.
\newblock {\em Duke Math. J.}, 3(3):443--454, 1937.

\bibitem{CF77}
W.~H. Cornish and P.~R. Fowler.
\newblock {Coproducts of De Morgan Algebras}.
\newblock {\em Bull. Austral. Math. Soc.}, 16:1--13, 1977.

\bibitem{DP02}
B.A Davey and H.A. Priestley.
\newblock {\em Introduction to Lattices and Order}.
\newblock Cambridge University Press, second edition, 2002.

\bibitem{G97}
S.~Ghilardi.
\newblock {Unification through Projectivity}.
\newblock {\em J. Logic Computat.}, 7(6):733--752, 1997.

\bibitem{H63}
P.~R. Halmos.
\newblock {\em Lectures in Boolean Algebras}.
\newblock Van Nostrand, 1963.

\bibitem{K58}
J.~A. Kalman.
\newblock {Lattices with Involution}.
\newblock {\em Trans. Amer. Math. Soc.}, 87:485--491, 1958.

\bibitem{ML98}
S.~MacLane.
\newblock {\em Categories for the Working Mathematician}.
\newblock Springer-Verlag, second edition, 1998.

\bibitem{MMT87}
R.~McKenzie, G.F. McNulty, and W.~Taylor.
\newblock {\em Algebras, Lattices, Varieties}.
\newblock Wadsworth and Brooks, 1987.

\bibitem{P70}
H.~A. Priestley.
\newblock {Representation of distributive lattices by means of ordered Stone
  spaces}.
\newblock {\em Bul. Lon. Math. Soc.}, 2(2):186--190, 1970.

\bibitem{S51}
R.~Sikorski.
\newblock {Homomorphisms, Mappings and Retracts}.
\newblock {\em Colloquium Math.}, 2:202--211, 1951.

\end{thebibliography}

\end{document}